\documentclass[a4paper,11pt]{article}

\usepackage{graphicx}
\usepackage{subcaption}
\usepackage{amsmath, amssymb, amsfonts, amsthm}
\usepackage{titlesec}
\usepackage{float}
\usepackage{enumitem}
\usepackage{url}
\usepackage{mathtools}
\usepackage{esint}
\usepackage{tikz}
\usepackage[utf8]{inputenc}
\usepackage[english]{babel}
\usepackage{geometry}
\usepackage{color}
\usepackage{epic}
\usepackage[color=white]{todonotes}

\usepackage{algorithm2e}

\newcommand{\R}{\mathbb{R}}
\newcommand{\C}{\mathbb{C}}
\newcommand{\T}{\mathcal{T}}

\newcommand{\B}{\mathcal{B}}

\newcommand{\CN}{\mathcal{N}}

\newcommand{\f}{\mathbf{f}}
\newcommand{\bfu}{\mathbf{u}}

\newcommand{\TauH}{\mathcal{T}_H}

\newcommand{\longdownarrow}{\rotatebox[origin=c]{-90}{$\hspace{3pt}\longrightarrow\hspace{3pt}$}}

\newcommand{\Kf}{\mathcal{K}} 
\newcommand{\Gf}{\mathcal{G}} 

\newcommand{\Wf}{\mathbf{W}}

\newcommand{\bfv}{\mathbf{v}}
\newcommand{\bfw}{\mathbf{w}}
\newcommand{\bfpsi}{\boldsymbol{\psi}}

\newcommand{\Hcurl}{\mathbf{H}(\curl)} 
\newcommand{\Hcurlhom}{\mathbf{H}_0(\curl)} 

\newcommand{\HNedhom}{\mathring{\mathcal{N}}(\mathcal{T}_H)}

\newcommand{\HLag}{\mathcal{S}(\mathcal{T}_H)}

 \newcommand{\quotes}[1]{``#1''}

\newcommand{\id}{I}
\DeclareMathOperator{\Int}{int}
\DeclareMathOperator{\Div}{div}
\DeclareMathOperator{\curl}{curl}

\DeclareMathOperator{\Grad}{grad}

\DeclareMathOperator{\ms}{ms}
\DeclareMathOperator{\corr}{corr}
\DeclareMathOperator{\interior}{in}
\DeclareMathOperator{\out}{out}
\definecolor{dark-green}{rgb}{0.0,0.4,0.0}

\newtheorem{theorem}{Theorem}[section]

\newtheorem{lemma}[theorem]{Lemma}

\theoremstyle{definition}

\newtheorem{remark}[theorem]{Remark}
\newtheorem{conclusion}[theorem]{Conclusion} %
\usepackage{framed}

\newenvironment{fshaded}{%
\MakeFramed {\FrameRestore}}%
{\endMakeFramed}

\newtheorem{cstep}{Step}

\begin{document}

\begin{center}
{\LARGE Computational homogenization of time-harmonic Maxwell's equations\renewcommand{\thefootnote}{\fnsymbol{footnote}}\setcounter{footnote}{0}
 \hspace{-3pt}\footnote{The authors acknowledge the support of the Brummer\&Partners MathDataLab and P. Henning also acknowledges the support of the Swedish Research Council (grant 2016-03339) and the G\"oran Gustafsson foundation.}}\\[2em]
\end{center}

\begin{center}
{\large Patrick Henning and Anna Persson \footnote[1]{Department of Mathematics, KTH Royal Institute of Technology, SE-100 44 Stockholm, Sweden.}}\\[2em]
\end{center}

\begin{center}
{\large{\today}}
\end{center}

\begin{center}
\end{center}

\begin{abstract}
In this paper we consider a numerical homogenization technique for curl-curl-problems that is based on the framework of the Localized Orthogonal Decomposition and which was proposed in [\emph{D.~Gallistl, P.~Henning, B.~Verf\"urth. SIAM J. Numer. Anal.~56-3:1570--1596, 2018}] for problems with essential boundary conditions. The findings of the aforementioned work establish quantitative homogenization results for the time-harmonic Maxwell's equations that hold beyond assumptions of periodicity, however, a practical realization of the approach was left open. In this paper, we transfer the findings from essential boundary conditions to natural boundary conditions and we demonstrate that the approach yields a computable numerical method. We also investigate how boundary values of the source term can effect the computational complexity and accuracy. Our findings will be supported by various numerical experiments, both in $2D$ and $3D$.
\end{abstract}


\vspace{1em}


\section{Introduction}\label{sec:intro}
In recent years, the interest in so called metamaterials has been growing rapidly. These materials are typically engineered on a nano- or microscale to exhibit unusual properties, such as band gaps and negative refraction \cite{Lamacz16,Schweizer17,Lipton18}. A common example is photonic crystals, which are composed of periodic cell structures interacting with light.

To model photonic crystals and other metamaterials properly, a numerical method that can handle such heterogeneous (or multiscale) media efficiently is needed. It is well known that classical finite element methods struggle to produce good approximations in this case, unless the mesh width is sufficiently small to resolve the microscopic variations intrinsic to the problem. Indeed, the mesh typically needs to be fine enough to resolve all features in the medium/metamaterial. This leads to issues with computational cost and available memory, which can be often only overcome by so-called multiscale methods. Multiscale methods are based on solving local decoupled problems on the microscale whose solutions are used to correct the original model and turn it into a homogenized macroscale model that can be solved cheaply. In the context of electromagnetic waves in multiscale media, corresponding methods were proposed in, e.g., \cite{Cao10, HocMS18, Verfuerth2017,Verfurth18, Verfurth19,Henning16_2, Stohrer14, Stohrer16, Stohrer17}. 

In \cite{Gallistl18} a method based on the Localized Orthogonal Decomposition (LOD) framework was proposed for multiscale curl-curl-problems with essential boundary conditions. The LOD technique was first introduced in \cite{Malqvist14} for elliptic equations in the $H^1$-space and was later developed to include many other types of equations and spaces, see, e.g., \cite{Abdulle17, Brown17, Hellman16, Persson16, Persson17, Persson18, Persson18_2, Peterseim17, Malqvist17} and the references therein. The technique relies on a splitting of the solution space into a coarse and a fine part, using a suitable (quasi-) interpolation operator. The coarse part is typically a classical finite element space and the fine part contains all the details not captured by the coarse space. The basis functions of the coarse space are then enhanced by computing so called correctors in the detail space. Due to an exponential decay of these correctors, they are quasi-local and can be computed efficiently by solving small problems. The resulting space (consisting of the corrected coarse basis functions) can be used in a Galerkin approach to find approximations that exhibit superior approximation properties in the energy norm when compared with the original coarse space. These results hold without assumptions on the structure of the problem, such as periodicity, scale separation or symmetry. 

Despite the wide range of successful applications, a construction of the LOD for $\Hcurl$-conforming N{\'e}d{\'e}lec finite element spaces turned out to be very challenging. The reason was the absence of suitable projection operators that are at the heart of the method (since corrections are always constructed in the kernel of such operators). The required projections need to be computable, local, $\Hcurl$-stable, and they need to commute with exterior derivatives. With the groundbreaking work by Falk and Winther \cite{FalkWinther2014}, such a projection was finally found and paved the way for extending the LOD framework to $\Hcurl$-spaces \cite{Gallistl18}. However, so far, the construction by Falk and Winther is only explicitly available for de Rham complexes with natural boundary conditions, whereas the theoretical results in \cite{Gallistl18} are obtained for curl-curl-problems with essential boundary conditions. This discrepancy (which turns out to be a non-trivial issue) prevented a practical implementation of the $\Hcurl$-LOD until now. Hence, the main objective of this paper is to explain how the method can indeed be used for practical computations and to demonstrate its performance in numerical experiments. Here we have several goals. First, we show how the $\Hcurl$-LOD can be extended from essential boundary conditions to natural boundary conditions. This enables us use the explicitly known Falk-Winther projection constructed in \cite{FalkWinther2014}, which we used in our implementation. For the arising method we prove that linear convergence is obtained in the $\Hcurl$-norm whenever the source term fulfills $\mathbf{f}\cdot \mathbf{n} = 0$. For general right hand sides $\mathbf{f}\in \mathbf{H}(\Div)$ we prove convergence of order $\sqrt{H}$. These reduced convergence rates for general $\mathbf{f}$ originate from the trace of the error on the domain boundary, which dominates the overall accuracy of the numerical method. This contribution can only be eliminated if the source term has a vanishing normal trace on $\partial \Omega$.
In order to restore a full linear convergence for general sources $\mathbf{f}$ we propose the usage of so-called source term correctors that can be used to make the error component on $\partial \Omega$ sufficiently small. As it turns out, such source term correctors can be also a helpful tool to solve the %
curl-curl-problem with essential boundary conditions since it allows the usage of the same Falk-Winther projection as for the case of natural boundary conditions. With this, we are able to practically solve the same types of problems as considered in \cite{Gallistl18}.
All the proposed methods were implemented and we include numerical examples in both $2D$ and $3D$ to support our theoretical findings. To the authors' knowledge this is also the first time that the edge-based Falk-Winther projection was implemented and the code for the projection (using the FEniCS software \cite{Fenics12}) is available at \cite{code}. Finally, we shall also demonstrate numerically that if we use a projection that is locally not commuting with exterior derivatives, then the arising $\Hcurl$-LOD suffers from reduced convergence rates. This shows that the required features of the projection are not only artifacts of the proof-technique but are central to obtain a converging method.

Another method based on the LOD framework was recently proposed in \cite{Ren19}. A different projection is constructed by utilizing the discrete Helmholtz decomposition of the fine detail space. Numerical experiments show that the proposed projection has the localization property, but an analytical proof of this property remains open.

The paper is organized as follows. In Section~\ref{sec:problem}, the curl-curl-problem is introduced and in Section~\ref{sec:LOD:natbc} the method for natural boundary conditions is described. Essential boundary conditions are briefly discussed in Section~\ref{sec:essential_bc}. Finally, numerical experiments are presented in Section~\ref{sec:experiments}. The construction of the Falk-Winther projection and a discussion of how to implement it is left to the appendix.

\section{Problem setting and notation}
\label{sec:problem}
Given a computational domain $\Omega \subset \R^3$, we consider the curl-curl-problem of the following form. Find $\bfu: \Omega \to \C^3$ such that
\begin{equation}\label{curlcurl_eq}
\begin{alignedat}{2}
\curl(\mu \curl \bfu) + \kappa \hspace{2pt}\bfu &= \f,& \quad &\text{in } \Omega, \\
\mu \curl \bfu \times \mathbf{n} &= 0,& &\text{on } \Gamma,
\end{alignedat}
\end{equation}
where $\mathbf{n}$ is the outward unit normal to $\Gamma:=\partial\Omega$. In the context of the time-harmonic Maxwell's equations, $\bfu=\mathbf{E}$ describes the (unknown) electric field and the right hand side $\f$ is a source term related to current densities. The coefficients $\mu$ and $\kappa$ are typically complex-valued and scalar parameter functions that describe material properties. In particular, they depend on electric permittivity and the conductivity of the medium through which the electric field propagates. We indirectly assume that the propagation medium is a highly variable \quotes{multiscale medium}, i.e., the coefficients $\mu$ and $\kappa$ are rapidly varying on a fine scale. In this work we mainly focus on natural boundary conditions (or Neumann-type boundary condition) which are of the form $\mu \curl \bfu \times \mathbf{n}=0$ on $\Gamma$. In the context of Maxwell's equations this corresponds to a boundary condition for (a 90 degree rotation of) the tangential component of the magnetic field.

In order to state a variational formulation of problem \eqref{curlcurl_eq}, we introduce a few spaces. For that purpose, let $G \subseteq \R^3$ denote an arbitrary bounded Lipschitz domain. We define the space of $\mathbf{H}(\curl)$-functions on $G$ by 
$$
\mathbf{H}(\curl, G):=\{ \mathbf{v}\in L^2(G, \C^3)|\curl \mathbf{v}\in L^2(G, \C^3)\}.
$$
The space is equipped with the standard inner product 
$$
(\bfv, \bfw)_{\mathbf{H}(\curl, G)}:=(\curl\bfv, \curl\bfw)_{L^2(G)}+(\bfv, \bfw)_{L^2(G)}
$$ 
where $(\cdot, \cdot)_{L^2(G)}$ is the complex $L^2$-inner product. Furthermore, we define the space of $\mathbf{H}(\Div)$-functions by
$$
\mathbf{H}(\Div, G):=\{\bfv\in L^2(G, \C^3)|\Div \bfv\in L^2(G, \C)\}
$$
with corresponding inner product 
$$
(\cdot, \cdot)_{\mathbf{H}(\Div, G)} := (\Div\bfv, \Div\bfw)_{L^2(G)}+(\bfv, \bfw)_{L^2(G)}.
$$
The restriction of $\mathbf{H}(\Div, G)$ to functions with a zero normal trace is given by
$$
\mathbf{H}_0(\Div, G):=\{\bfv\in \mathbf{H}(\Div, G)|\hspace{3pt} \bfv\cdot \mathbf{n} \vert_{\partial G} =0\}.
$$
If we drop the dependency on $G$, we always mean the corresponding space with $G=\Omega$. With this we consider the following variational formulation of problem \eqref{curlcurl_eq}: find $\bfu \in \Hcurl$ such that
\begin{align}\label{curlcurl_weak_natural}
\B(\bfu,\bfv) := (\mu \curl \bfu, \curl \bfv)_{L^2(\Omega)} + (\kappa \bfu, \bfv)_{L^2(\Omega)} = (\f,\bfv)_{L^2(\Omega)}, \quad \mbox{for all } \bfv \in \Hcurl.
\end{align}
Observe that the natural boundary condition $\mu \curl \bfu \times \mathbf{n}=0$ on $\Gamma$ is intrinsically incorporated in the variational formulation. 

For problem \eqref{curlcurl_weak_natural} to be well-defined, we shall also make a set of assumptions on the data, which read as follows.
\begin{enumerate}[label=(A\arabic*)]
	\item $\Omega \subset \R^3$ is an open, bounded, contractible polyhedron with a Lipschitz boundary. 
	\item The coefficients fulfill $\mu \in L^{\infty}(\Omega, \R^{3 \times 3})$ and $\kappa \in L^{\infty}(\Omega, \C^{3 \times 3})$.  
	\item The source term fulfills $\f \in \mathbf{H}(\Div)$.
	\item The coefficients are such that the sesquilinear form 
$\B: \mathbf{H}(\curl)\times \mathbf{H}(\curl) \to \C$ given by \eqref{curlcurl_weak_natural}
is elliptic on $\mathbf{H}(\curl)$, i.e., there exists $\alpha>0$ such that
$$
|\B(\bfv, \bfv)|\geq \alpha\|\bfv\|^2_{\mathbf{H}(\curl)}
 \quad\text{for all }\bfv\in\mathbf{H}(\curl) .
$$
We assume that this property is independent of the computational domain $\Omega$ (i.e., we can restrict the sesquilinear form to subdomains of $\Omega$ and have still ellipticity).
\end{enumerate}
Under the above assumptions (A1)-(A4), we have that \eqref{curlcurl_weak_natural} is well-defined by the Lax-Milgram-Babu{\v{s}}ka theorem \cite{Bab70fem}. Note that none of the assumptions is very restrictive and that they are fulfilled in many physical setups (cf. \cite{Gallistl18,FR05maxwell} for discussions and examples). In \cite{Verfuerth2017} it was also demonstrated that the coercivity can be relaxed to a frequency-dependent  inf-sup-stability condition for indefinite problems. However, in order to avoid the additional technical details that would come with the indefinite setting, we will only focus on the simpler $\mathbf{H}(\curl)$-elliptic case as given by (A4).

\section{Numerical homogenization for natural boundary conditions}
\label{sec:LOD:natbc}
In this section, we follow the arguments presented in \cite{Gallistl18} for the case of essential boundary conditions (i.e., $\bfu \times \mathbf{n} = 0$ on $\Gamma$) to transfer them to the case of natural boundary conditions as given in \eqref{curlcurl_eq}.

\subsection{Basic discretization and multiscale issues}
Let us start with some basic notation that is used for a conventional discretization of \eqref{curlcurl_eq}. This setting will be also used to exemplify why such a standard discretization of multiscale curl-curl-problems is problematic.

Throughout this paper, we make the following assumptions on the discretization.
\begin{enumerate}[label=(B\arabic*)]
	\item $\TauH$ is a shape-regular and quasi-uniform partition  of $\Omega$, in particular we have $\cup\TauH=\overline{\Omega}$.
	\item The elements of $\TauH$ are (closed) tetrahedra and are such that any two distinct elements $T, T'\in \TauH$ are either disjoint or share a common vertex, edge or face.
	\item The mesh size (i.e., the maximum diameter of an element of $\TauH$) is denoted by $H$.
\end{enumerate}
On $\TauH$ we consider lowest-order $H^1$-, $\mathbf{H}(\Div)$- and $\mathbf{H}(\curl)$-conforming elements. Here, $\HLag\subset H^1(\Omega)$ denotes the space of $\TauH$-piecewise affine and continuous functions (i.e., scalar-valued first-order Lagrange finite elements) and we set $\mathring{\mathcal{S}}(\TauH):=\HLag \cap H^1_0(\Omega)$ as the subset with a zero trace on $\partial \Omega$.
The lowest order N{\'e}d{\'e}lec finite element (cf.\ \cite[Section 5.5]{Monk}) is denoted by
\[
  \CN(\TauH):=\{\bfv\in \mathbf{H}(\curl)|\forall T\in \TauH: \bfv|_T(\mathbf{x})=\mathbf{a}_T\times\mathbf{x}+\mathbf{b}_T \text{ with }\mathbf{a}_T, \mathbf{b}_T\in\C^3\}.
\]
Later on, we will also need the lowest-order Raviart--Thomas finite element space, which is given by
$$
  \mathcal{RT}(\TauH):=\{\bfv\in \mathbf{H}(\Div)|\forall T\in \TauH: \bfv|_T(\mathbf{x})=a_T\hspace{2pt}\mathbf{x}+\mathbf{b}_T \text{ with } a_T\in \C, \mathbf{b}_T\in\C^3\}.
$$
With this, a standard $\Hcurl$-conforming discretization of problem \eqref{curlcurl_weak_natural} is to find $\bfu_H \in \CN(\TauH)$ such that
\begin{align}\label{curlcurl_discrete_H_natural}
\B(\bfu_H,\bfv_H) = (\f,\bfv_H)_{L^2(\Omega)}, \quad \mbox{for all } \bfv \in \CN(\TauH).
\end{align}
As a Galerkin approximation, $\bfu_H$ is a quasi-best approximation in the $\Hcurl$-norm and fulfills the classical estimate
\begin{align}
\label{classical-estimate}
\| \bfu - \bfu_H \|_{\Hcurl} \le C
\inf_{\mathbf{v}_H \in \mathcal{N}(\TauH)} \| \bfu - \mathbf{v}_H \|_{\Hcurl} \le C H \left( \| \bfu \|_{H^1(\Omega)} + \| \curl \bfu \|_{H^1(\Omega)} \right).
\end{align}
A serious issue is faced if $\mu$ and $\kappa$ are multiscale coefficients that are rapidly oscillating with a characteristic length scale of the oscillations that is of order $\delta \ll 1$. In this case, both $\| \bfu \|_{H^1(\Omega)}$ and $\| \curl \bfu \|_{H^1(\Omega)}$ will blow up, even for differentiable coefficients $\mu$ and $\kappa$. Even worse, if $\mu$ and $\kappa$ are discontinuous, then the required regularity $\bfu, \curl \bfu  \in H^1(\Omega,\C^3)$ might not be available at all (cf. \cite{Cost90regmaxwellremark, CDN99maxwellinterface, BGL13regularitymaxwell}). In such a setting the estimate \eqref{classical-estimate} becomes meaningless, as convergence rates can become arbitrarily slow. On top of that, convergence can only be observed in the fine scale regime $H<\delta$. This imposes severe restrictions on the mesh resolution and the costs for solving \eqref{curlcurl_discrete_H_natural} can become tremendous. It is worth to mention that the picture does not change if we look at the error in the $L^2$-norm. This is because $\| \bfu - \bfu_H \|_{L^2(\Omega)}$ and $\| \bfu - \mathbf{u}_H \|_{\Hcurl}$ are typically of the same order, due to the large kernel of the curl-operator. In particular, for large mesh sizes $H$ the space $\mathcal{N}(\TauH)$ does not contain any good $L^2$-approximations of $\bfu$. This is because fast variations in $\kappa$ cause equally fast oscillations in $\bfu$ with an amplitude of order $\mathcal{O}(1)$. These oscillations cannot be captured on coarse meshes and due to their non-negligible amplitude they are crucial in the $L^2$-sense.
For periodic coefficients, the appearance of such oscillations can be explained with analytical homogenization theory (cf. \cite{Henning16_2}). An example for a function $\bfu$ that illustrates this aspect is given in Figure \ref{fig:ref_solution} below.
Note that this is a crucial difference to $H^1$-elliptic problems, where good $L^2$-approximations on coarse meshes are at least theoretically available (though they are typically not found by standard Galerkin methods). A very good review about this topic for $H^1$-elliptic problems is given in \cite{Peterseim16}.

\subsection{De Rham complex and subscale correction operators}

As discussed in the previous subsection, a conventional discretization of the curl-curl-problem \eqref{curlcurl_weak_natural} is problematic as it requires very fine meshes $\TauH$ and hence very high-dimensional discrete spaces. Motivated by this issue, the question posed in \cite{Gallistl18} is whether it is possible to construct a {\it quasi-local} sub-scale correction operator 
$\Kf$ that acts on $\Hcurl$ and which only depends on $\B(\cdot,\cdot)$ (but not on the source term $\f$), such that the original multiscale problem can be replaced by a numerically homogenized equation. By \quotes{numerically homogenized equation} we mean an equation that can be discretized with a Galerkin method in the standard (coarse) space $\mathcal{N}(\TauH)$, such that the corresponding (homogenized) solution $\bfu_H^0 \in \mathcal{N}(\TauH)$ is a good coarse scale approximation of $\bfu$ (in the dual space norm $\| \cdot \|_{\mathbf{H}(\Div)^{\prime}}$, cf. \cite{Gallistl18} for details) and such that $\bfu_H^0 + \Kf( \bfu_H^0)$ (i.e., homogenized solution plus corrector) yields a good approximation of $\bfu$ in the $\Hcurl$-norm. The (numerically) homogenized equation takes the form
$$
\B(  (I + \Kf) \bfu_H^0 , (I + \Kf)\bfv  ) = (\f , (I + \Kf)\bfv  )_{L^2(\Omega)} \qquad \mbox{for all } \bfv \in \mathcal{N}(\TauH)
$$
and provided that $\Kf$ is available, this can be solved cheaply in a coarse (and hence low-dimensional) space $\mathcal{N}(\TauH)$. In \cite{Gallistl18}, such a homogenized problem was obtained for essential boundary conditions using the technique of Localized Orthogonal Decompositions (LOD).

The central tool in the construction of a suitable corrector operator $\Kf$ is the existence of local and stable projection operators that commute with the exterior derivative. For that we apply the theory of finite element exterior calculus (FEEC), where we refer to the excellent survey papers \cite{ArnoldFalkWinther2006,ArnoldFalkWinther2010} for an introduction to the topic. A cochain complex of vector spaces is essentially a sequence of homomorphisms between these spaces with the property that the image of each of the homomorphisms is a subset of the kernel of the next homomorphism. If the vector spaces are the spaces of differential $p$-forms on some smooth manifold and if the homomorphisms are given by the exterior derivative, we call the cochain complex a {\it de Rham complex}. In the following,
we consider the $L^2$ de Rham complex with natural boundary conditions which is of the form
\begin{align*}
H^1(\Omega) \hspace{5pt} \overset{\Grad}{\longrightarrow}  \hspace{5pt} \mathbf{H}(\curl,\Omega) \hspace{5pt} \overset{\curl}{\longrightarrow}  \hspace{5pt}  
\mathbf{H}(\Div,\Omega) \hspace{5pt} \overset{\Div}{\longrightarrow}  \hspace{5pt}  L^2(\Omega).
\end{align*}
Obviously, we have for the exterior derivative $\curl \circ \Grad = \mathbf{0}$ and $\Div \circ \curl = 0$, which shows the central property of a cochain complex (that is the inclusion of the image of each homomorphism into the kernel of the next). We also consider the
corresponding finite element subcomplex (of lowest order spaces)
\begin{align*}
\HLag \hspace{5pt} \overset{\Grad}{\longrightarrow}  \hspace{5pt} \CN(\TauH) \hspace{5pt} \overset{\curl}{\longrightarrow}  \hspace{5pt}  
\mathcal{RT}(\TauH) \hspace{5pt} \overset{\Div}{\longrightarrow}  \hspace{5pt}  \mathcal{P}(\TauH),
\end{align*}
where $\mathcal{P}(\TauH)\subset L^2(\Omega)$ is the space of $\TauH$-piecewise constant functions. In this setting (and in more general form), Falk and Winther \cite{FalkWinther2014} proved the existence of stable and local projections ${\pi}_H$ between any Hilbert space of the de Rham complex and the corresponding discrete space of the finite element subcomplex, so that the following diagram commutes:
\begin{align*}
\begin{matrix}
H^1(\Omega) & \overset{\Grad}{\longrightarrow}  & \mathbf{H}(\curl,\Omega) & \overset{\curl}{\longrightarrow}  &  
\mathbf{H}(\Div,\Omega) & \overset{\Div}{\longrightarrow}  &  L^2(\Omega) \\
\longdownarrow \hspace{3pt}\mbox{\small${\pi}_H^V$}
& &\longdownarrow \hspace{3pt}\mbox{\small${\pi}_H^E$} & & \longdownarrow \hspace{3pt}\mbox{\small${\pi}_H^F$}
& &
\longdownarrow \hspace{3pt}\mbox{\small${\pi}_H^T$}\\
\HLag  & \overset{\Grad}{\longrightarrow}  & \CN(\TauH)  & \overset{\curl}{\longrightarrow}  & 
\mathcal{RT}(\TauH)  & \overset{\Div}{\longrightarrow}  & \mathcal{P}(\TauH).
\end{matrix}
\end{align*}
Here, $V$ stands for vertex, $E$ stands for edges, $F$ for faces and $T$ for tetrahedra. 
In particular (and for our setting very essential) is the commutation 
\begin{align}
\label{commuting-property-edge-node}
 \pi_H^E( \nabla v ) =  \nabla \pi_H^V( v ) \qquad \mbox{for all } v \in H^1(\Omega).
\end{align}
The operators are local (i.e., local information can only spread in small nodal environments) and they admit local stability estimates, for $T\in \TauH$, of the form
\begin{align}
\label{falk-winther-global-interpol-est}
\| {\pi}_H(v) \|_{L^2(T)} \lesssim \| v \|_{L^2(U(T))} + H | v |_{H\Lambda(U(T))}
\quad \mbox{and}
\quad | {\pi}_H(v) |_{H\Lambda(T)} \lesssim  | v |_{H\Lambda(U(T))}
\end{align}
for all $v \in H\Lambda(\Omega)$, where
$H\Lambda(\Omega)$ stands for either of the spaces $H^1(\Omega)$, $\mathbf{H}(\curl,\Omega)$ or $\mathbf{H}(\Div,\Omega)$, the semi-norm is given by $| v |_{H\Lambda} = \| d v \|_{L^2}$ where $d$ stands for the exterior derivative on $H\Lambda(\Omega)$ and ${\pi}_H$ stands for the corresponding projection operator from the commuting diagram. The constant $C>0$ in the estimate is generic and the patch $U(K)$ is a small environment of $K$ that consists of elements from $\TauH$ and which has a diameter of order $H$. We refer the reader to \cite[Theorem 2.2]{FalkWinther2014} for further details on this aspect.

The idea proposed in \cite{Gallistl18} is now to correct/enrich the conventional discrete spaces $\Lambda_H= \mbox{image}(\pi_H)$ by information from the kernel of the projections. Due to the direct and stable decomposition 
$$H\Lambda(\Omega) = \Lambda_H(\Omega) \oplus \Wf(\Omega),$$ 
where $\Wf(\Omega) := \mbox{kern}(\pi_H)$, we know that the exact solution $\bfu$ to the curl-curl-problem \eqref{curlcurl_weak_natural} can be uniquely decomposed into a coarse contribution in $\Lambda_H(\Omega) = \mathcal{N}(\TauH)$ and a fine (or detail) contribution in $\Wf:=\Wf(\Omega)=\mbox{kern}(\pi_H^E)$, i.e.,
\begin{align}
\label{preliminary-decomp}
\bfu = \pi_H^E(\bfu) + (\bfu - \pi_H^E(\bfu)  ).
\end{align}
In order to identify and characterize the components of this decomposition, a {\it Green's corrector operator} 
\begin{align*}
\Gf: \mathbf{H}(\curl)^{\prime} \rightarrow \Wf
\end{align*}
can be introduced. For $\mathbf{F} \in \mathbf{H}(\curl)^\prime$ it is defined by the equation
\begin{align}
\label{cor-greens-op}
\B(\hspace{2pt}\Gf(\mathbf{F}) \hspace{2pt} , \bfw )=\mathbf{F}(\bfw)\qquad \mbox{for all } \bfw\in \Wf.
\end{align}
It is well-defined by the Lax-Milgram-Babu{\v{s}}ka theorem due to assumption (A4) and since $\Wf$ is a closed subspace as the kernel of a $\Hcurl$-stable projection. The main feature of the Green's corrector is that it allows us to characterize the $\Hcurl$-stable decomposition of the exact solution $\bfu\in\Hcurl$ to problem \eqref{curlcurl_weak_natural} as
\begin{align}
\label{decomposition-exact-solution}
\bfu = \bfu_H - (\Gf \circ \mathcal{L})(\bfu_H) + \Gf(\f),
\end{align}
where $\bfu_H:=\pi_H^E(\bfu) \in \mathcal{N}(\TauH)$ is the coarse part. The operator $\mathcal{L}$ is the differential operator associated with $\B(\cdot,\cdot)$, i.e., $\mathcal{L}(\bfv) :=\B(\bfv , \cdot)$.

To see that \eqref{decomposition-exact-solution} holds, we start from the unique decomposition \eqref{preliminary-decomp} and insert it into \eqref{curlcurl_weak_natural}, where we restrict the test functions to $\Wf\subset \Hcurl$. This yields
\begin{align*}
\B(  \bfu - \pi_H^E(\bfu) , \bfw )= - \B(   \pi_H^E(\bfu)  , \bfw ) + (\f, \bfw)_{L^2(\Omega)} 
= - \B( (\Gf \circ \mathcal{L})(\bfu_H) , \bfw ) + \B( \Gf(\f) , \bfw )
\end{align*}
for all $\bfw \in \Wf$. Since both $\bfu - \pi_H^E(\bfu)$ and $\Gf(\f) - (\Gf \circ \mathcal{L})(\bfu_H)$ are elements of $\Wf$, they must be identical, proving \eqref{decomposition-exact-solution}.

\subsection{Idealized multiscale method for curl-curl-problems}
In the last subsection we saw that the exact solution can be decomposed into $\bfu = \bfu_H - (\Gf \circ \mathcal{L})(\bfu_H) + \Gf(\f)$, where $\bfu_H:=\pi_H^E(\bfu)$. This motivates to define the ideal corrector operator $\Kf: \mathcal{N}(\TauH) \rightarrow \Wf$ as 
$$
\Kf :=  - \Gf \circ \mathcal{L}
$$
and to consider a numerical homogenized equation of the form: find $\bfu_H^0 \in \mathcal{N}(\TauH)$ such that
\begin{align}
\label{ideal-LOD-form}
\B(  (I + \Kf) \bfu_H^0 , (I + \Kf)\bfv  ) = (\f , (I + \Kf)\bfv  )_{L^2(\Omega)} \qquad \mbox{for all } \bfv \in \mathcal{N}(\TauH).
\end{align}
Since this is a Galerkin approximation, C\'ea's lemma implies that $(I + \Kf) \bfu_H^0$ is an $\Hcurl$-quasi best approximation of $\bfu$ in $(I + \Kf)\mathcal{N}(\TauH)$ and hence
\begin{align*}
\| \bfu -   (I + \Kf) \bfu_H^0 \|_{\Hcurl} \lesssim
\| \bfu -   (I + \Kf) \bfu_H \|_{\Hcurl} \overset{\eqref{decomposition-exact-solution}}{=} \| \Gf(\f) \|_{\Hcurl}.
\end{align*}
In fact, if $\mu$ and $\kappa$ are self-adjoint, it can be shown that $\bfu_H^0=\pi_H^E(\bfu)$ and hence the error has the exact characterization $\bfu -   (I + \Kf) \bfu_H^0=\Gf(\f)$ (this can be proved analogously to \cite[Theorem 10]{Gallistl18}).

As we just saw, in order to quantify the discretization error $\mathbf{e}_H=\bfu -   (I + \Kf) \bfu_H^0$, all we have to do is to estimate $\Gf(\f)$. Since $\Omega$ is a contractible domain, we know that $\Gf(\f) \in\Hcurl$ admits a regular decomposition (cf. \cite{Hiptmair2002,HiptmairXu2007}) of the form
\begin{align}
\label{regular-decomposition}
\Gf(\f) = \mathbf{z} + \nabla \theta, \qquad \mbox{where } \enspace \mathbf{z}\in [H^1(\Omega)]^3 \enspace \mbox{ and } \enspace \theta \in H^1(\Omega),
\end{align}
and with
\begin{align}
\label{stabiliy-regular-decomposition}
\|  \mathbf{z} \|_{H^1(\Omega)} + \| \theta \|_{H^1(\Omega)} \lesssim \| \Gf(\f) \|_{\Hcurl}.
\end{align}
Exploiting the commuting property of the Falk-Winther projections, we obtain that we can write $\Gf(\f) \in \Wf$ as
\begin{align}
\label{identity-Greens-corrector-f}
\Gf(\f) = \Gf(\f) - \pi_H^E(\Gf(\f))
\overset{\eqref{commuting-property-edge-node}}{=} (\mathbf{z}  - \pi_H^E(\mathbf{z})) + \nabla ( \theta - \pi_H^V(\theta)).
\end{align}
By the Bramble-Hilbert lemma together with the local stability estimates \eqref{falk-winther-global-interpol-est} we easily see that for any $T \in \TauH$ it holds 
\begin{align}\label{falk-winther-global-interpol-est-decop-1}
\| \mathbf{z}  - \pi_H^E(\mathbf{z}) \|_{L^2(T)} \le C H \| \mathbf{z} \|_{H^1(U(T))} 
\end{align}
and analogously
\begin{align}\label{falk-winther-global-interpol-est-decop-2}
\|  \theta - \pi_H^V(\theta) \|_{L^2(T)} \le C H \|  \theta \|_{H^1(U(T))}.
\end{align}
Again, $U(T)$ is a generic environment of $T$ that consists of few (i.e., $\mathcal{O}(1)$) elements from $\TauH$. Observe that this yields a local regular decomposition of interpolation errors in the spirit of Sch\"oberl (cf. \cite[Theorem 1]{Sch08aposteriori} and \cite[Lemma 4]{Gallistl18}).

Next, we use the identity \eqref{identity-Greens-corrector-f} in the definition of the Green's corrector $\Gf$ to obtain
\begin{align}
\label{ideal-method-est-0}\alpha \| \Gf(\f) \|_{\Hcurl}^2 &\le \B( \Gf(\f) , \Gf(\f) )
= (\f , \mathbf{z}  - \pi_H^E(\mathbf{z}) ) + (\f , \nabla ( \theta - \pi_H^V(\theta)) )_{L^2(\Omega)} \\
\nonumber&= (\f , \mathbf{z}  - \pi_H^E(\mathbf{z}) ) - ( \Div \f , \theta - \pi_H^V(\theta) )_{L^2(\Omega)} 
+ ( \f \cdot \mathbf{n}, \theta - \pi_H^V(\theta) )_{L^2(\Gamma)}. 
\end{align}
Using the interpolation error estimates \eqref{falk-winther-global-interpol-est-decop-1} and \eqref{falk-winther-global-interpol-est-decop-2} for the Falk-Winther projections, we immediately have
\begin{align}
\label{ideal-method-est-1}
(\f , \mathbf{z}  - \pi_H^E(\mathbf{z}) ) 
\lesssim H \| \f \|_{L^2(\Omega)} \| \mathbf{z} \|_{\Hcurl} \overset{\eqref{stabiliy-regular-decomposition}}{\lesssim}
H \| \f \|_{L^2(\Omega)}  \| \Gf(\f) \|_{\Hcurl}.
\end{align}
and
\begin{align}
\label{ideal-method-est-2}
- ( \Div \f , \theta - \pi_H^V(\theta) )_{L^2(\Omega)}
\lesssim H \|  \Div \f \|_{L^2(\Omega)} \| \theta \|_{H^1(\Omega)} \overset{\eqref{stabiliy-regular-decomposition}}{\lesssim}
H \|  \Div \f  \|_{L^2(\Omega)}  \| \Gf(\f) \|_{\Hcurl}.
\end{align}
It remains to estimate the last term. For that, let $\mathcal{F}_{H,\Gamma}$ denote the set of faces (from our triangulation) that are located on the domain boundary $\Gamma$. We obtain with the local trace inequality
\begin{eqnarray}
\nonumber\lefteqn{| ( \f \cdot \mathbf{n}, \theta - \pi_H^V(\theta) )_{L^2(\Gamma)}|
 \le \sum_{S\in \mathcal{F}_{H,\Gamma}}  \| \f \cdot \mathbf{n} \|_{L^2(S)}
 \| \theta - \pi_H^V(\theta)  \|_{L^2(S)} } \\
\nonumber &\le& \| \f \cdot \mathbf{n} \|_{L^2(\Gamma)}
 \left(  \sum_{S \in \mathcal{F}_{H,\Gamma}} \| \theta - \pi_H^V(\theta)  \|_{L^2(S)}^2 \right)^{1/2} \\
\nonumber &\lesssim&  \| \f \cdot \mathbf{n} \|_{L^2(\Gamma)} 
  \left(  \underset{T\cap \Gamma \in \mathcal{F}_{H,\Gamma}}{\sum_{T \in \TauH}} H^{-1} \| \theta - \pi_H^V(\theta)  \|_{L^2(T)}^2 
  + H  \| \nabla \theta - \nabla \pi_H^V(\theta)  \|_{L^2(T)}^2  \right)^{1/2} \\
\nonumber  &\overset{\eqref{falk-winther-global-interpol-est},\eqref{falk-winther-global-interpol-est-decop-2}}{\lesssim}&
   \| \f \cdot \mathbf{n} \|_{L^2(\Gamma)} 
  \left(  \underset{T\cap \Gamma \in \mathcal{F}_{H,\Gamma}}{\sum_{T \in \TauH}} H \| \nabla \theta \|_{L^2(U(T))}^2 
  + H  \| \nabla \theta \|_{L^2(U(T))}^2  \right)^{1/2}\\
 \label{ideal-method-est-3} &\overset{\eqref{stabiliy-regular-decomposition}}{\lesssim}& \sqrt{H}  \hspace{2pt} \| \f \cdot \mathbf{n} \|_{L^2(\Gamma)}   \| \Gf(\f) \|_{\Hcurl}.
\end{eqnarray}
Here we used the regularity of the mesh $\TauH$ and that each $T$-neighborhood $U(T)$ contains only $\mathcal{O}(1)$ elements. 
Combining the estimates \eqref{ideal-method-est-0}, \eqref{ideal-method-est-1}, \eqref{ideal-method-est-2} and \eqref{ideal-method-est-3}, we obtain the following conclusion.
\begin{conclusion}
\label{conclusion-estimate-ideal-method}
Let $\bfu \in \Hcurl$ denote the exact solution to curl-curl-problem \eqref{curlcurl_discrete_H_natural} with natural boundary conditions and let $\bfu_H^0 \in \mathcal{N}(\TauH)$ denote the solution to ideal numerically homogenized equation \eqref{ideal-LOD-form}, then it holds the error estimate
\begin{align*}
\| \bfu -   (I + \Kf) \bfu_H^0 \|_{\Hcurl} \lesssim H \hspace{2pt} \| \f \|_{\mathbf{H}(\Div)} + \sqrt{H} \hspace{2pt} \| \f \cdot \mathbf{n} \|_{L^2(\Gamma)}.
\end{align*}
In particular, if $\f \in \mathbf{H}_0(\Div)$ (i.e., it has a vanishing normal trace), then we have optimal order convergence with
\begin{align*}
\| \bfu -   (I + \Kf) \bfu_H^0 \|_{\Hcurl} \lesssim H \hspace{2pt} \| \f \|_{\mathbf{H}(\Div)}.
\end{align*}
\end{conclusion}
From Conclusion \ref{conclusion-estimate-ideal-method} we see that in general we expect the error to be dominated by the term $\sqrt{H} \hspace{2pt} \| \f \cdot \mathbf{n} \|_{L^2(\Gamma)}$. Note that in the case of essential boundary conditions, the problematic term $( \f \cdot \mathbf{n}, \theta - \pi_H^V(\theta) )_{L^2(\Gamma)}$ will always vanish since the corresponding regular decomposition guarantees that $\theta = 0$ on $\Gamma$ (cf. \cite{Hiptmair2002,Sch08aposteriori}).

Before we discuss how to restore the linear convergence for source terms with arbitrary normal trace, we need to discuss how to localize the corrector operator $\mathcal{K}$.

\subsection{Localized multiscale method for curl-curl-problems}
\label{subsection:LOD_Hcurl}
From a computational point of view, the ideal multiscale method \eqref{ideal-LOD-form} is not feasible, as it requires the computation of global correctors $\mathcal{K}(\bfpsi_{E})$ for each nodal basis function $\bfpsi_E$ of the N{\'e}d{\'e}lec space $\mathcal{N}(\TauH)$. The equation that describes a corrector  $\mathcal{K}(\bfpsi_{E}) \in \Wf$ reads
\begin{align*}
\B ( \hspace{2pt} \mathcal{K}(\bfpsi_{E}) \hspace{2pt}, \bfw ) = - \B ( \bfpsi_E , \bfw ) \qquad \mbox{for all } \bfw \in \Wf.
\end{align*}
This problem has the same computational complexity as the original curl-curl-problem. Solving it for every basis function would hence multiply the original computing costs. To avoid this issue, we first split the global corrector $\mathcal{K}$ into a set of element correctors $\mathcal{K}_T$. For a tetrahedron $T\in \TauH$, the corresponding element corrector $\mathcal{K}_T(\bfv_H) \in \Wf$ is defined as 
\begin{align}
\label{ideal-element-corrector}
\B ( \hspace{2pt} \mathcal{K}_{T}(\bfv_H) \hspace{2pt}, \bfw ) = - \B_T ( \bfv_H , \bfw ) \qquad \mbox{for all } \bfw \in \Wf,
\end{align}
where $\B_T(\bfv ,\bfw ):=(\mu \curl \bfv ,\curl \bfw )_{L^2(T)} + (\kappa \hspace{2pt} \bfv , \bfw )_{L^2(T)}$. Obviously, we have 
\begin{align}
\label{relation-corrector-element-correctors}
\Kf(\bfv_H) = \sum_{T\in \TauH} \mathcal{K}_{T}(\bfv_H).
\end{align}
The advantage of this formulation is that we can now exploit the strong localization of the element correctors $\mathcal{K}_{T}$. In fact, they show an extremely fast decay to zero away from the element $T$. The decay can be quantified and is exponential in units of $H$. This is a consequence of the local support of the source term (i.e., $\B_T(\cdot,\cdot)$) and the fact that $\Kf_T$ maps into the kernel of the Falk-Winther projection. This allows to restrict the computational domain in \eqref{ideal-element-corrector} to small environments $\textup{N}^m(T)$ of $T$ that have a diameter of order $m\cdot H$. The decay was first proved in \cite[Theorem 14]{Gallistl18} for the case of essential boundary conditions (i.e., for $\mathcal{K}(\bfpsi_{E}) \times \mathbf{n} = 0$ on $\Gamma$). For natural boundary conditions we can proceed similarly (with minor modifications) to obtain the following result.
\begin{lemma}\label{lemma:approximation-of-correctors}
Given $T\in \TauH$, we define the (open) $m$-layer neighborhood recursively by
\begin{align*}
\textup{N}^0(T):=\Int(T) \qquad \mbox{and} \qquad \textup{N}^m(T):=
\Int\left(\bigcup\hspace{2pt}\{K\in \TauH \hspace{2pt}| \hspace{2pt}K\cap\overline{\textup{N}^{m-1}(T)}\neq \emptyset\}\right).
\end{align*}
On these patches (for $m>0$) we define the truncated element correctors 
$$
\mathcal{K}_{T,m}(\bfv_H) \in \Wf( \hspace{2pt}\textup{N}^m(T) \hspace{2pt})
:= \{ \bfw \in \Wf | \hspace{2pt}  \bfw \equiv 0 \hspace{4pt}\mbox{\rm in } \Omega \setminus \textup{N}^m(T) \}
$$
as solutions to 
\begin{align}
\label{truncated-element-corrector}
\B ( \hspace{2pt} \mathcal{K}_{T,m}(\bfv_H) \hspace{2pt}, \bfw ) = - \B_T ( \bfv_H , \bfw ) \qquad \mbox{for all } \bfw \in \Wf( \hspace{2pt}\textup{N}^m(T) \hspace{2pt})
\end{align}
For $\bfv_H \in \mathcal{N}(\TauH)$, an approximation of the global corrector is given (in the spirit of \eqref{relation-corrector-element-correctors}) by
$$
\mathcal{K}_{m}(\bfv_H):=\sum_{T\in\TauH} \mathcal{K}_{T,m}(\bfv_H).
$$
The error between the truncated approximation $\mathcal{K}_{m}$ and the ideal corrector $\mathcal{K}$ can be bounded by
\begin{align*}
\| \mathcal{K}(\bfv_H) - \mathcal{K}_{m}(\bfv_H) \|_{\Hcurl} \le C \rho^m \| \bfv_H \|_{\Hcurl},
\end{align*}
where $0<\rho<1$ and $C>0$ are generic constants that are independent of $m$, $H$ or the speed of the variations of $\mu$ and $\kappa$. 
\end{lemma}
\begin{remark}[Boundary conditions of the truncated element correctors]
Observe that the truncated element correctors $\mathcal{K}_{T,m}(\bfv_H)$ given by \eqref{truncated-element-corrector} have mixed boundary conditions. On $\partial \textup{N}^m(T) \setminus \Gamma$ they fulfill essential boundary conditions, i.e., $\mathcal{K}_{T,m}(\bfv_H) \times \mathbf{n} =0$, and on $\partial \textup{N}^m(T) \cap \Gamma$ they fulfill natural boundary conditions, i.e., $\mu \curl \mathcal{K}_{T,m}(\bfv_H) \times \mathbf{n}=0$.
This is in contrast to the ideal element correctors given by \eqref{ideal-element-corrector} which only involve natural boundary conditions.
\end{remark}
From Lemma \ref{lemma:approximation-of-correctors} we see that that we can approximate the global corrector $\mathcal{K}$ efficiently by computing a set of element correctors $\mathcal{K}_{T,m}$ with a small local support $\textup{N}^m(T)$. With this we can formulate the following main theorem on the overall accuracy of localized multiscale scheme (which belongs to the class of LOD methods, cf. \cite{Gallistl18,Henning14,Malqvist14}).
\begin{theorem}\label{theorem-local-LOD}
Let $\bfu \in \Hcurl$ denote the exact solution to the curl-curl-problem with natural boundary conditions. Furthermore, let $\mathcal{K}_m$ denote the corrector approximation from Lemma \ref{lemma:approximation-of-correctors} and let $\bfu_{H,m} \in \mathcal{N}(\TauH)$ denote the unique solution to the numerically homogenized equation
\begin{align*}
\B( (I + \Kf_m ) \bfu_{H,m} , (I + \Kf_m ) \bfv_H ) = (\f , (I + \Kf_m )\bfv_H  )_{L^2(\Omega)}
\qquad \mbox{for all }  \bfv_H \in \mathcal{N}(\TauH).
\end{align*} 
Then the following error estimates holds
\begin{align}
\label{full-error-estimate-natbou}
\| \bfu - (\bfu_{H,m} + \mathcal{K}_m(\bfu_{H,m})) \|_{\Hcurl} \lesssim
(H + \rho^m) \hspace{2pt} \| \f \|_{\mathbf{H}(\Div)} + \sqrt{H} \hspace{2pt} \| \f \cdot \mathbf{n} \|_{L^2(\Gamma)}.
\end{align}
In particular, if $m\gtrsim |\log(H)|/|\log(\rho)|$ and $\f \in \mathbf{H}_0(\Div)$, then we have optimal order convergence in $H$, i.e., it holds
\begin{align*}
\| \bfu - (\bfu_{H,m} + \mathcal{K}_m(\bfu_{H,m})) \|_{\Hcurl} \lesssim
H \hspace{2pt} \| \f \|_{\mathbf{H}(\Div)} .
\end{align*}
\end{theorem}
\begin{proof}
The proof is analogous to \cite[Conclusion 15]{Gallistl18}. Let $\bfu_H= \pi_H^E(\bfu) \in \mathcal{N}(\TauH)$ be the Falk-Winther projection of the exact solution. Then we have with Lemma \ref{lemma:approximation-of-correctors}
\begin{eqnarray*}
\lefteqn{\| \bfu - \bfu_H - \mathcal{K}_m(\bfu_H) \|_{\Hcurl}
\overset{\eqref{decomposition-exact-solution}}{\le} \| \Gf(\f) \|_{\Hcurl} + \| \mathcal{K}(\bfu_H)- \mathcal{K}_m(\bfu_H) \|_{\Hcurl}} \\
&\le & \| \Gf(\f) \|_{\Hcurl} + C \rho^m \| \pi_H^E(\bfu) \|_{\Hcurl}  \overset{\eqref{falk-winther-global-interpol-est}}{\le} 
 \| \Gf(\f) \|_{\Hcurl} + C \rho^m \| \bfu \|_{\Hcurl} \\
 &\le&  \| \Gf(\f) \|_{\Hcurl} + C \rho^m \| \f \|_{\mathbf{H}(\Div)}. 
\end{eqnarray*}
Conclusion \ref{conclusion-estimate-ideal-method} together with the quasi-best approximation property of Galerkin methods finishes the proof.
\end{proof}
Theorem \ref{theorem-local-LOD} shows that we can still achieve the same order of accuracy as for the ideal multiscale method \eqref{ideal-LOD-form} if the number of layers $m$ is selected proportional to $|\log(H)|$. Practically this means that in order to compute a sufficiently accurate corrector $\Kf_m$, we need to solve local problems on small patches with a diameter of order $H |\log(H)|$. The number of these local problems equals the number of all tetrahedra times the number of edges of a tetrahedron (which is equal to the number of basis functions of $\mathcal{N}(\TauH)$ that have support on a tetrahedron $T\in \TauH$). Hence, the total number of local problems is  $6\cdot \sharp \TauH$. All these problems are small, cheap to solve and they are fully independent from each other, which allows for parallelization in a straightforward way. 

There are two more issues to be discussed, first, we need a computational representation of the Falk-Winther projection operator $\pi_H^E$. As this summarizes the results of the previous works \cite{FalkWinther2014,Gallistl18} and in particular \cite{FalkWinther2015}, we postpone it to the appendix. The second issue that remains is the question if and how we can restore a full linear convergence for the error in \eqref{full-error-estimate-natbou}, if $\f$ does not have a vanishing normal trace. We will investigate this question in the next subsection.

\subsection{Source term and boundary corrections}
\label{subsection:source-corrections}

In Theorem \ref{theorem-local-LOD} we have seen that we can in general not expect a full linear convergence in the mesh size $H$, unless $\f \in \mathbf{H}_0(\Div)$. The estimate \eqref{full-error-estimate-natbou} suggests a convergence of order $H^{1/2}$, which we can also confirm numerically (cf. Section \ref{sec:example2}). Hence, as the full linear rate is typically not achieved for arbitrary source terms, this poses the question if we can modify the method accordingly.

One popular way to handle the influence of dominating source terms or also general (inhomogeneous) boundary conditions is to compute source and boundary corrections (cf. \cite{Henning14} for details in the context of $H^1$-elliptic problems). Here we recall decomposition \eqref{decomposition-exact-solution}, which says that we can write the exact solution as
$$
\bfu = \bfu_H + \Kf(\bfu_H) + \Gf(\f).
$$
Practically, we can approximate $\Gf(\f)$ in the spirit of Lemma \ref{lemma:approximation-of-correctors} by a corrector $\Gf_m(\f)$ based on localization. This means, that for each $T\in\TauH$, we can solve for an element source corrector $\Gf_{T,m}(\f) \in \Wf( \hspace{2pt}\textup{N}^m(T) \hspace{2pt})$ with
\begin{align*}
\B ( \hspace{2pt} \Gf_{T,m}(\f)  \hspace{2pt}, \bfw ) = ( \f , \bfw )_{L^2(T)} \qquad \mbox{for all } \bfw \in \Wf( \hspace{2pt}\textup{N}^m(T) \hspace{2pt}).
\end{align*}
The global source corrector is then given by
$$
\Gf_{m}(\f):=\sum_{T\in\TauH} \Gf_{T,m}(\f).
$$
After this, we can solve the corrected homogenized problem which is of the following form: find $\bfu_{H,m}^{\corr} \in \mathcal{N}(\TauH)$ such that
$$
\B( (I + \Kf_m) \bfu_{H,m}^{\corr}  , (I+\Kf_m) \bfv_H  ) = (\f , (I+\Kf_m) \bfv_H )_{L^2(\Omega)} - 
\B(  \Gf_{m}(\f) , (I+\Kf_m) \bfv_H  ) 
$$
for all $\bfv_H \in \mathcal{N}(\TauH)$. The final approximation is given by
$$
\bfu_{H,m}^{\ms} := (I + \Kf_m) \bfu_{H,m}^{\corr}  + \Gf_{m}(\f)
$$
and it is possible to prove that the error converges exponentially fast in $m$ to zero:
\begin{align*}
\| \bfu -  \bfu_{H,m}^{\ms}  \|_{\Hcurl} \lesssim \rho^m \| \f  \|_{L^2(\Omega)},\qquad \mbox{for a generic }  0<\rho < 1.
\end{align*}
The error estimate can be proved using similar arguments as in the proof of Theorem \ref{theorem-local-LOD}. Here we recall that $\rho$ is independent of $m$ and $H$. By selecting $m\gtrsim k |\log(H)|/|\log(\rho)|$, we can achieve the convergence rate $H^k$ for any $k\ge 1$. For very large $m$ (i.e $\textup{N}^m(T)=\Omega$) the method is exact.

\begin{remark}
Note that this strategy can be also valuable for handling general non-homogenous boundary conditions. For example, considering the boundary condition $\mu \curl \bfu \times \mathbf{n} = \mathbf{g}$ on $\Gamma$ (which can be straightforwardly incorporated into the variational formulation), it is possible to compute boundary correctors to increase the accuracy. This works analogously to the case of source correctors, i.e., one needs to compute element boundary correctors $\Gf_E(\mathbf{g})\in \Wf$ for each boundary edge $E\subset \Gamma$. The edge boundary correctors solve problems of the form 
$$
\B( \Gf_E(g) , \bfw ) = ( \mathbf{g} , \bfw )_{L^2(E)} \qquad \mbox{for all } \bfw \in \Wf. 
$$
\end{remark}
Though the above strategy can potentially yield any degree of accuracy, it involves a  computational overhead, since it additionally requires to compute the whole set of element source correctors $\Gf_{T,m}(\f)$. One way to reduce this overhead is to only compute correctors $\Gf_{T,m}(\f)$ for which $T$ is close to the boundary. For example, we can set
$$
\mathcal{T}_H^{\interior} := \{ T \in \TauH| \hspace{3pt} \overline{\textup{N}^m(T)} \cap \Gamma \not= \emptyset \} \qquad \mbox{and} \qquad
\mathcal{T}_H^{\out} :=  \TauH \setminus \mathcal{T}_H^{\interior}.
$$
In this case we have for the error
\begin{eqnarray*}
\lefteqn{ \bfu - \big(\bfu_H + \Kf_m(\bfu_H) + \sum_{T\in \mathcal{T}_H^{\out}} \Gf_{T,m}(\f) \big)} \\
&=& (\Kf_m -\Kf)\bfu_H
+ \sum_{T\in \mathcal{T}_H^{\interior}} (\Gf_{T,m}-\Gf_{T}) \f
+ \sum_{T\in \mathcal{T}_H^{\out}} (\Gf_{T,m}-\Gf_{T})\f
- \sum_{T\in \mathcal{T}_H^{\interior}} \Gf_{T,m}(\f).
\end{eqnarray*}
Here, the first three terms show an exponential convergence with order $\rho^m$, whereas the last term is not polluted by the boundary influence and converges with order $H$, since $( \sum_{T\in \mathcal{T}_H^{\interior}} \Gf_{T,m}(\f))_{\vert \Gamma}=0$. 

\section{Numerical homogenization for essential boundary conditions}\label{sec:essential_bc}
In the following we want to consider the curl-curl-problem with essential boundary conditions as originally done in \cite{Gallistl18}, i.e., we seek $\bfu^0: \Omega \to \C^3$ such that
\begin{equation}\label{curlcurl_eq_essential}
\begin{alignedat}{2}
\curl(\mu \curl \bfu^0) + \kappa \hspace{2pt}\bfu^0 &= \f,& \quad &\text{in } \Omega, \\
\bfu^0 \times \mathbf{n} &= 0,& &\text{on } \partial\Omega.
\end{alignedat}
\end{equation}
For the variational formulation we denote the restriction of $\Hcurl$ to functions with a vanishing normal trace on $\Gamma$ by
$$\mathbf{H}_0(\curl):=\{\bfv\in \mathbf{H}(\curl,\Omega)|\hspace{3pt} \bfv\times \mathbf{n} \vert_{\partial \Gamma} =0\}.$$
Hence, the weak formulation of \eqref{curlcurl_eq_essential} becomes: find $\bfu^0 \in \mathbf{H}_0(\curl)$ with
$$
\B( \bfu^0 , \bfv ) = (\f , \bfv)_{L^2(\Omega)} \qquad\mbox{for all } \bfv \in \mathbf{H}_0(\curl).
$$ 
Assuming again (A1)-(A4), this problem is well-posed. 

To discretize the problem, let the subspace of $\CN(\TauH)$ which contains functions with a homogenous tangential trace to be given by
$$
 \mathring{\CN}(\TauH):= \CN(\TauH) \cap \Hcurlhom.
$$
With this a multiscale approximation analogously to the case of natural boundary conditions (cf. Section \ref{sec:LOD:natbc}) can be obtained by considering Falk-Winther projections between the spaces with vanishing traces (cf. \cite{Gallistl18,FalkWinther2014}), i.e., local and stable projections
$$
\mathring{\pi}_H^{E} : \mathbf{H}_0(\curl) \rightarrow \mathring{\CN}(\TauH) 
\qquad
\mbox{and}
\qquad
\mathring{\pi}_H^{V} : H^1_0(\Omega) \rightarrow \mathring{\mathcal{S}}(\TauH):=\mathcal{S}(\TauH) \cap H^1_0(\Omega)
$$
with the commuting property $\nabla \mathring{\pi}_H^{V}(v) = \mathring{\pi}_H^{E}(\nabla v)$ for all $v \in H^1_0(\Omega)$. In this setting, the following result can be proved (see \cite{Gallistl18}):

Recall the notation from Section \ref{subsection:LOD_Hcurl}. For $T\in \TauH$ and $m\in \mathbb{N}_{>0}$, let the local detail space be given by
$$
\mathring{\Wf}( \hspace{2pt}\textup{N}^m(T) \hspace{2pt}) := \{ \bfw \in \mbox{kern}(\mathring{\pi}_H^{E}) | 
 \hspace{2pt}  \bfw \equiv 0 \hspace{4pt}\mbox{\rm in } \Omega \setminus \textup{N}^m(T) \} 
 \subset  \mathbf{H}_0(\curl)
$$
and the element correctors $\mathring{\Kf}_{T,m}(\bfv_H) \in \mathring{\Wf}( \hspace{2pt}\textup{N}^m(T) \hspace{2pt})$ 
 be the solutions to 
\begin{align}
\label{KfTm-essent-bound}
 \B ( \mathring{\Kf}_{T,m}(\bfv_H) , \bfw ) = \B_T(\bfv_H , \bfw  ) \qquad \mbox{for all } \bfw\in \mathring{\Wf}( \hspace{2pt}\textup{N}^m(T) \hspace{2pt}).
\end{align}
We define the global corrector as usual by $\mathring{\Kf}_m(\bfv_H):=\sum_{T\in \TauH}  \mathring{\Kf}_{T,m}(\bfv_H)$. With this, the final multiscale approximation is given by $\mathring{\bfu}_{H,m} \in \mathring{\CN}(\TauH)$ with
\begin{align}
\label{lod-essential-boundary-conditions}
\B( (\id + \mathring{\Kf}_m ) \mathring{\bfu}_{H,m} ,  (\id + \mathring{\Kf}_m )\bfv ) =
(\f , (\id + \mathring{\Kf}_m )\bfv ) \qquad \mbox{for all } \bfv \in \mathring{\CN}(\TauH)
\end{align}
and it holds the error estimate 
\begin{align*}
\| \bfu - (\id + \mathring{\Kf}_m ) \mathring{\bfu}_{H,m} \|_{\Hcurl} \lesssim (H + \rho^m) \| \f \|_{\mathbf{H}(\Div)},
\end{align*}
where $0<\rho<1$ is a generic decay constant as before. Note that this result holds independently of the normal trace of $\f$. This is due to the regular decomposition of functions $\bfv \in \mathbf{H}_0(\curl)$, which is of the form $\bfv = \mathbf{z} + \nabla \theta$, where $\mathbf{z}\in [H^1_0(\Omega)]^3$ and $\theta \in H^1_0(\Omega)$. Hence, the problematic term $( \f \cdot \mathbf{n}, \theta)_{L^2(\Gamma)}$ in \eqref{ideal-method-est-0} (which cause the degeneration of convergence rates) is equal to zero.

The main issue with essential boundary conditions is the availability of the required Falk-Winther projection $\mathring{\pi}_H^{E}$ (in the de Rham complex with vanishing traces). So far, the projection has not yet been explicitly constructed, which is also the main reason why the approach of \cite{Gallistl18} was not yet numerically validated. It is worth to mention that a corresponding construction would be similar to the construction in the case of natural boundary conditions (see Appendix \ref{section-falk-winther-construction}), with the difference that all degrees of freedom on the boundary need to be set to zero in the construction (cf. the descriptions in \cite{FalkWinther2014} and in \cite[Section 7.1]{Demlow2017}). To our understanding, this leads to local equations (as part of the construction of the projection) which are formulated on a complex with mixed boundary conditions, which we however had problems to solve numerically. Hence, an implementation of $\mathring{\pi}_H^{E}$ remains open, which raises the question if it is possible to work with the projection $\pi_H^{E}$ on the full space instead.

In \cite[Section 8]{Gallistl18} it was suggested to start from $\pi_H^{E}$ and then set all weights to zero that belong to edge-based basis functions $\bfpsi_E$ for a boundary edge $E\subset \Gamma$. This means the multiscale method is based on the kernel of the operator
$$
\pi_H^{0,E} := \mathcal{I}_H^{0,E} \circ \pi_H^E \hspace{3pt}: \hspace{3pt}\Hcurlhom \rightarrow \HNedhom,
$$
where $\mathcal{I}_H^{0,E}: \Hcurlhom \rightarrow \HNedhom$ is the standard edge-based N{\'e}d{\'e}lec interpolation operator. This choice works indeed well if $\f \in \mathbf{H}(\Div)$ has a vanishing normal trace, but it fails for general $\f$ (i.e., if $\f\cdot\mathbf{n} \not=0$ on $\Gamma$). The reason is that $\pi_H^{0,E}$ does no longer fulfill the very important commuting diagram property in a \quotes{one-coarse layer}-environment of $\Gamma$. This leads to a local loss of stability. In the arising multiscale method, it can be compensated if $\f \cdot \mathbf{n}=0$. However, in any other case the stability issues on the boundary cause a pollution of the numerical solution. We added a corresponding numerical experiment in Section~\ref{sec:example4} which shows that the convergence rate is no longer linear but only $\mathcal{O}(H^{1/2})$.

$\\$
As an alternative we suggest to simply keep on using the same $\pi_H^E$ as for the case of natural boundary conditions and proceed otherwise analogously to \eqref{lod-essential-boundary-conditions}. Practically, this means that we select the local detail space for the computation of $\Kf_{T,m}^0$ (analogously to \eqref{KfTm-essent-bound}) as
$$
\Wf^0( \hspace{2pt}\textup{N}^m(T) \hspace{2pt}) := \{ \bfw \in \Hcurlhom | 
 \quad  \pi_H^E(\bfw)=0 
 \quad \mbox{and} \quad \bfw \equiv 0 \hspace{4pt}\mbox{\rm in } \Omega \setminus \textup{N}^m(T) \}.
$$
Otherwise, nothing changes and the multiscale approximation is given by $\bfu_{H,m}^0 \in \mathring{\CN}(\TauH)$ with
\begin{align}
\label{lod-essential-boundary-conditions-with-org-fwp}
\B( (\id + \Kf^0_m ) \bfu_{H,m}^0 ,  (\id + \Kf^0_m )\bfv ) =
(\f , (\id + \Kf^0_m )\bfv ) \qquad \mbox{for all } \bfv \in \mathring{\CN}(\TauH).
\end{align}
Since typically $\pi_H^E(\bfv) \not\in \HNedhom$ for $\bfw\in \Hcurlhom$, the modified formulation introduces an error at the boundary which reduces the overall convergence rate. This is similar to the case of natural boundary conditions with $\f \cdot \mathbf{n} \not= 0$. Indeed, the experiments in Section~\ref{sec:example3} and Section~\ref{sec:example4} show that we get full linear convergence in the case when $\f \cdot \mathbf{n} = 0$ and convergence of order $\sqrt{H}$ when $\f \cdot \mathbf{n} \not= 0$. Again, a full linear rate can be reconstructed by using source term corrections close to the boundary as described in Section \ref{subsection:source-corrections}. This is investigated numerically in Section~\ref{sec:example4}. Here we stress that the use of source correctors to handle the case of essential boundary conditions with $\f \cdot \mathbf{n} \not= 0$ is currently the only available option for optimal convergence rates. A preferable alternative would be a multiscale construction based on a computable Falk-Winther projection for essential boundary conditions as this would not require additional source correctors. However, as explained above, such a projection is currently not available.

\section{Numerical Experiments}\label{sec:experiments}
In this section we present some numerical examples in both $2D$ and $3D$ to verify the theoretical findings. For simplicity, we choose the computational domain to be either the unit square or the unit cube, depending on the dimension of the problem. In $2D$, the considered curl-curl-problem is defined analogously to the $3D$ problem given by \eqref{curlcurl_weak_natural}, however, with the difference that the $2$-dimensional $\curl$-operator maps vector fields to scalar values via $\curl(\bfv):= \partial_{x_1} \bfv_2 - \partial_{x_2} \bfv_1$. The Falk-Winther projection in $2D$ and $3D$ is discussed in Appendix \ref{section-falk-winther-construction}. All other modifications in the problem and method formulation are straightforward. Even though there are no analytical results for the $2D$ case, our numerical experiments will show that the convergence behavior is analogous to the theoretically supported $3D$ setting.

To compute the error of the method in the examples below, we compute a reference solution on a finer (uniform) mesh, of size $h = \sqrt{2}\cdot 2^{-6}$ in $2D$ and $h=\sqrt{3}\cdot 2^{-4}$ in $3D$, using classical FEM in the lowest order N\'{e}d\'{e}lec space $\mathcal{N}(\mathcal{T}_h)$. The relative errors are reported in the energy norm induced by the sesquilinear form $\B(\cdot,\cdot)$. That is, if $\mathbf{e}$ denotes the difference between the full LOD multiscale approximation $\mathbf{u}_{H,m}+\mathcal{K}_m(\mathbf{u}_{H,m})$ and the reference solution $\mathbf{u}_h$, then the relative error is computed as
\begin{align*}
\Bigg(\frac{\B(\mathbf{e},\mathbf{e})}{\B(\mathbf{u}_h,\mathbf{u}_h)}\Bigg)^{1/2} = \Bigg(\frac{(\mu \curl \mathbf{e}, \curl \mathbf{e}) + (\kappa \mathbf{e}, \mathbf{e})}{(\mu \curl \mathbf{u}_h, \curl \mathbf{u}_h) + (\kappa \mathbf{u}_h, \mathbf{u}_h)}\Bigg)^{1/2}.
\end{align*}

To achieve a multiscale character of the problem, we choose the coefficients, $\mu$ and $\kappa$, to be checkerboards. The values in the cubes (squares in $2D$) are set to $1$ (black) and $0.001$ (white), respectively. See Figure~\ref{fig:checkerboards} for a visualization. The size of the checkerboard (i.e., the size of each constant block) is always chosen to be twice the size of the reference mesh. This means that the reference solution always resolves the coefficients. A typical reference solution in the $2D$ case is plotted in Figure~\ref{fig:ref_solution}, where we clearly see the multiscale structure from the checkerboard.  

All implementations are made using the FEniCS software \cite{Fenics12}. The code for the Falk-Winther projection in both $2D$ and $3D$ can be obtained from \cite{code}.

\begin{figure}[h]
	\centering
	\begin{subfigure}{0.35\textwidth}
		\includegraphics[width=\textwidth]{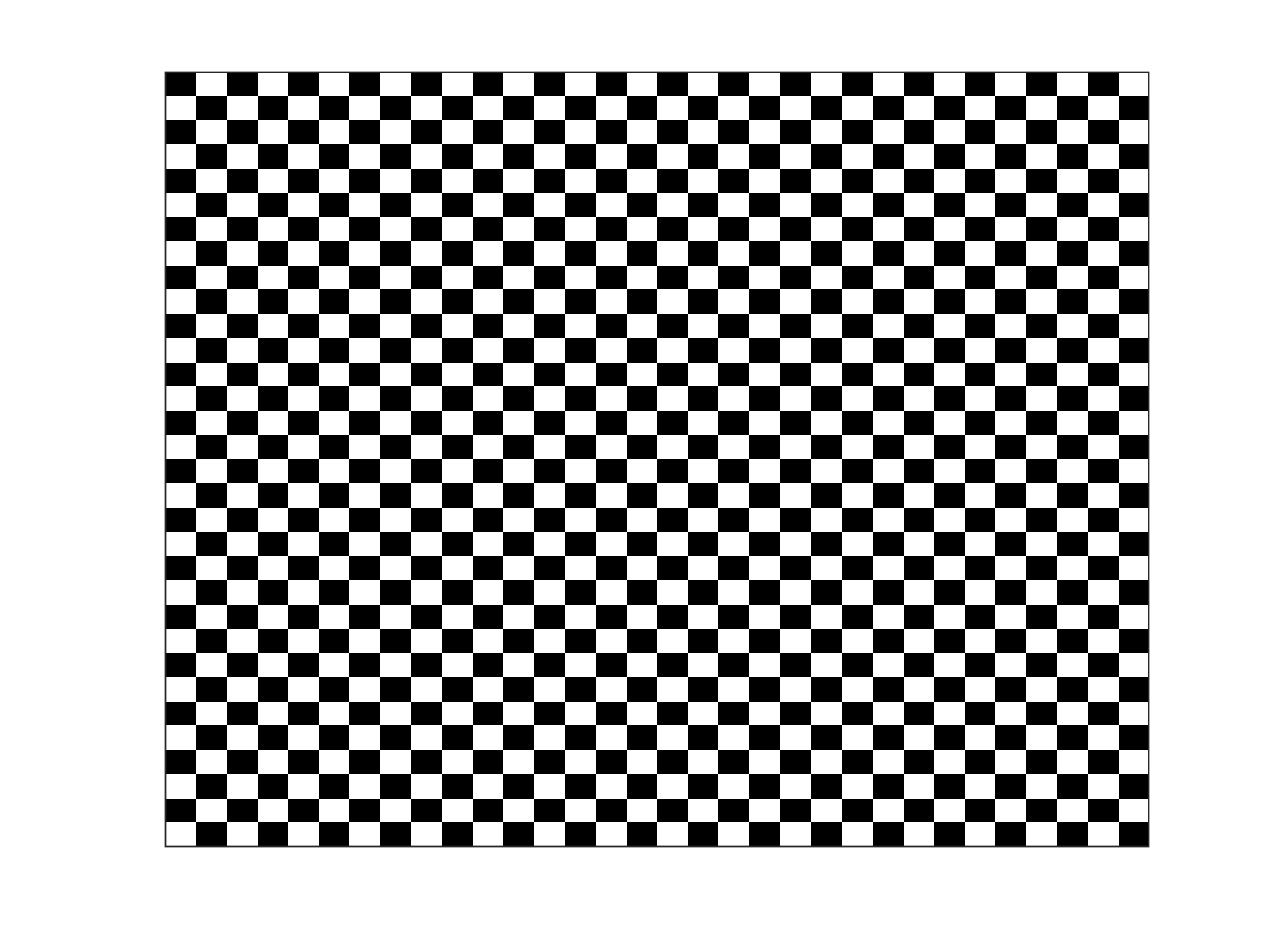}
		\caption{Two dimensions}
	\end{subfigure} ~
	\begin{subfigure}{0.35\textwidth}
		\includegraphics[width=\textwidth]{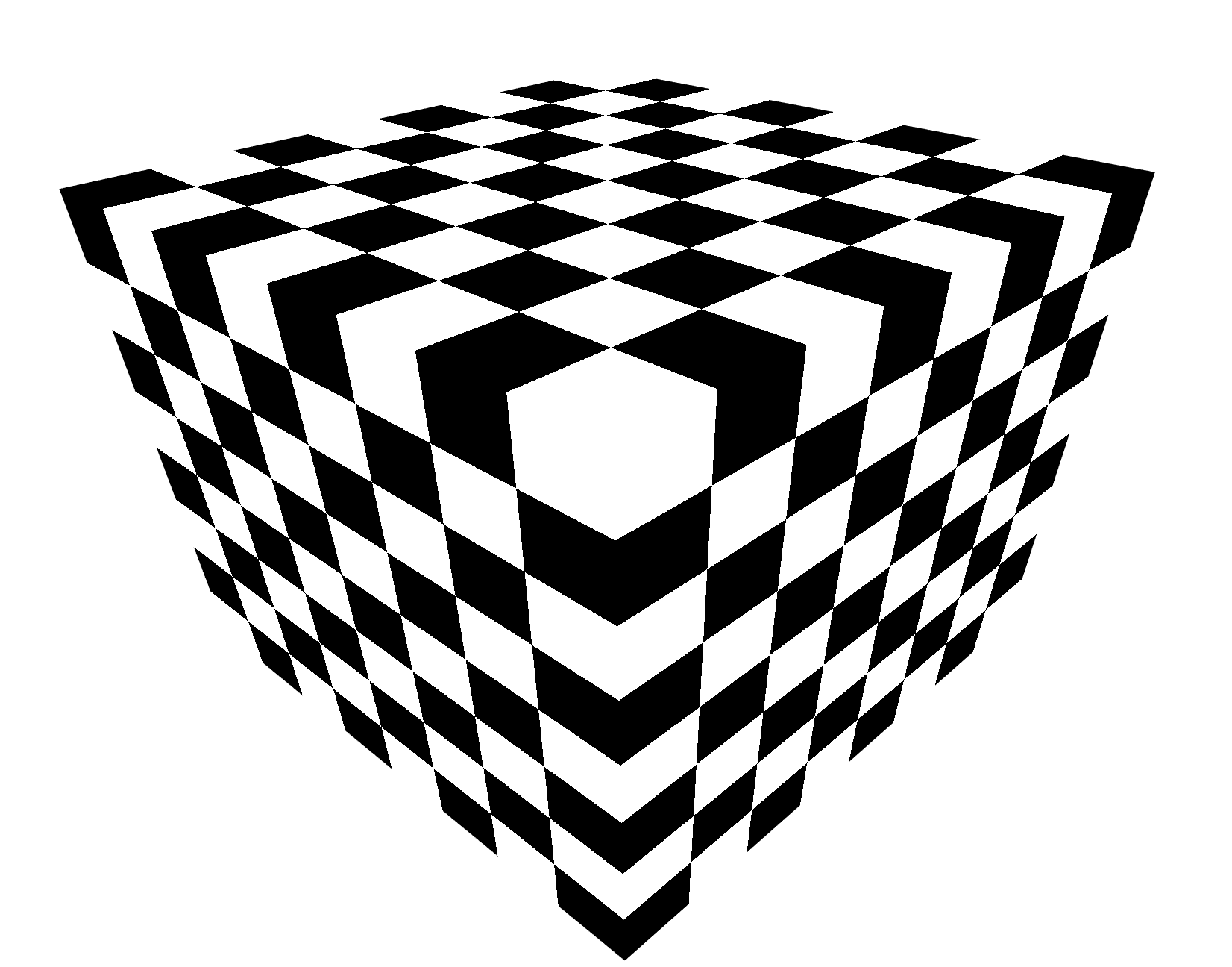}
		\caption{Three dimensions}
	\end{subfigure}
	\caption{Checkerboard coefficient for two and three dimensions.}\label{fig:checkerboards}
\end{figure}

\begin{figure}[h]
	\centering
	\includegraphics[width=0.6\textwidth]{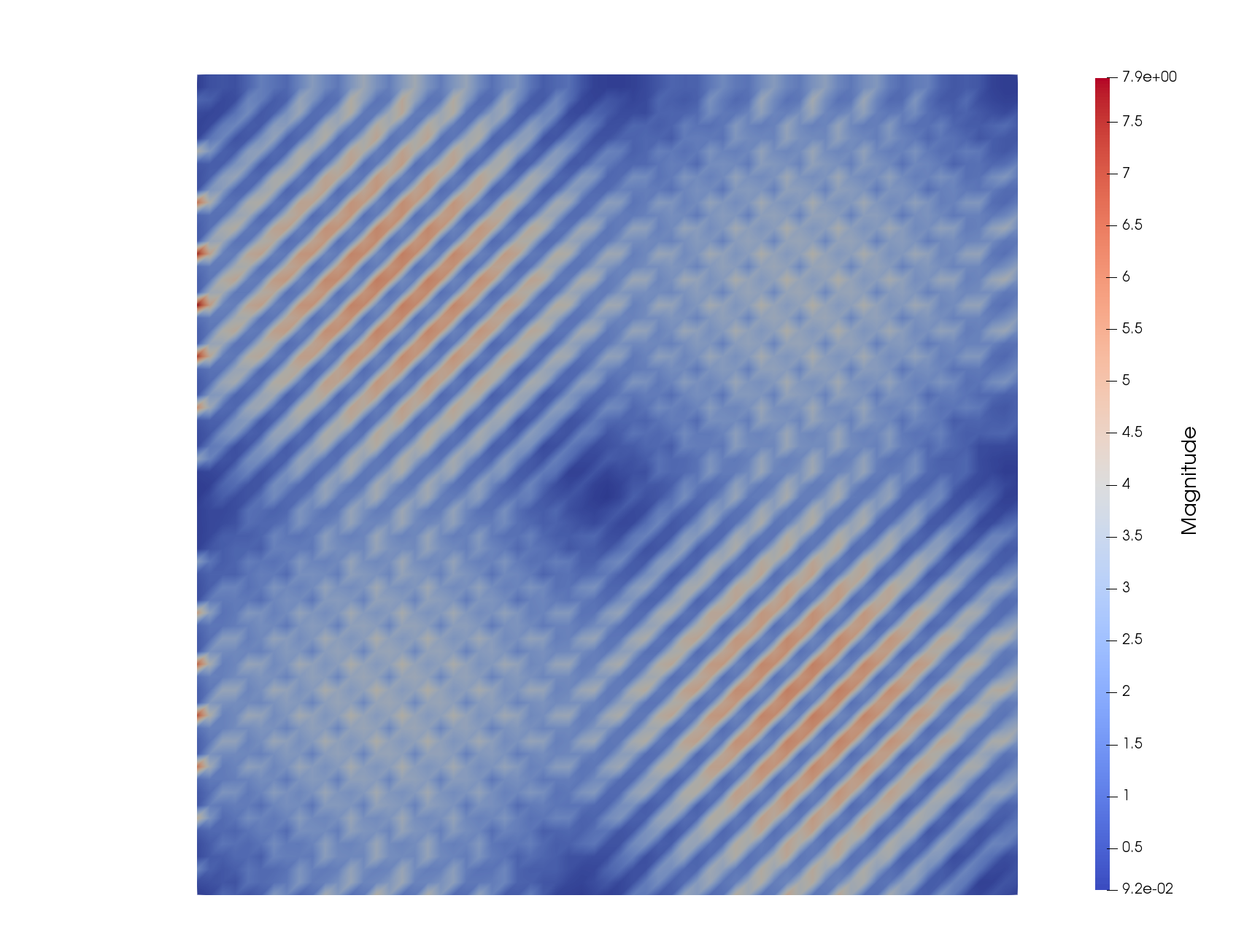}
	\caption{Reference solution in $2D$ with $\mathbf{f}(x,y) = [\sin(2\pi x),\sin(2 \pi y)]$ and natural boundary conditions.}\label{fig:ref_solution}
\end{figure}

\subsection{Example 1 - Natural boundary conditions}\label{sec:example1}
In this example we show that we obtain linear convergence of the method for problems that fulfill $\mathbf{f} \cdot \mathbf{n} = 0$, as predicted in Theorem~\ref{theorem-local-LOD}. 

In $2D$ we choose $\mathbf{f}(x,y) = [\sin(2\pi x),\sin(2 \pi y)]$ and we compute LOD approximations on uniform meshes of size $H = \sqrt{2}\cdot 2^{-j}$, $j=0,...,5$. Furthermore, we let the number of layers in the patch $N^m(T)$ range through $m = 1,1,2,2,3,4$ as $H$ decreases. For each value of $H$, the error for the LOD method is compared to the classical finite element approximation in the N\'{e}d\'{e}lec space $\mathcal{N}(\TauH)$. 

In $3D$ we choose $\mathbf{f}(x,y,z) = [-x(x-1)(2z-1),0,z(z-1)(2x-1)]$. The LOD and the corresponding FEM approximations are computed for $H=\sqrt{3}\cdot2^{-j}$, $j=0,..,3$. 

The results are plotted in Figure~\ref{fig:Example1_2D}, where we clearly see the linear convergence (marked by a dashed line) and also that the method produces significantly smaller errors than the classical FEM on the same coarse mesh $\TauH$. 

\begin{figure}[h]
	\centering
	\begin{subfigure}{0.45\textwidth}
		\includegraphics[width=\textwidth]{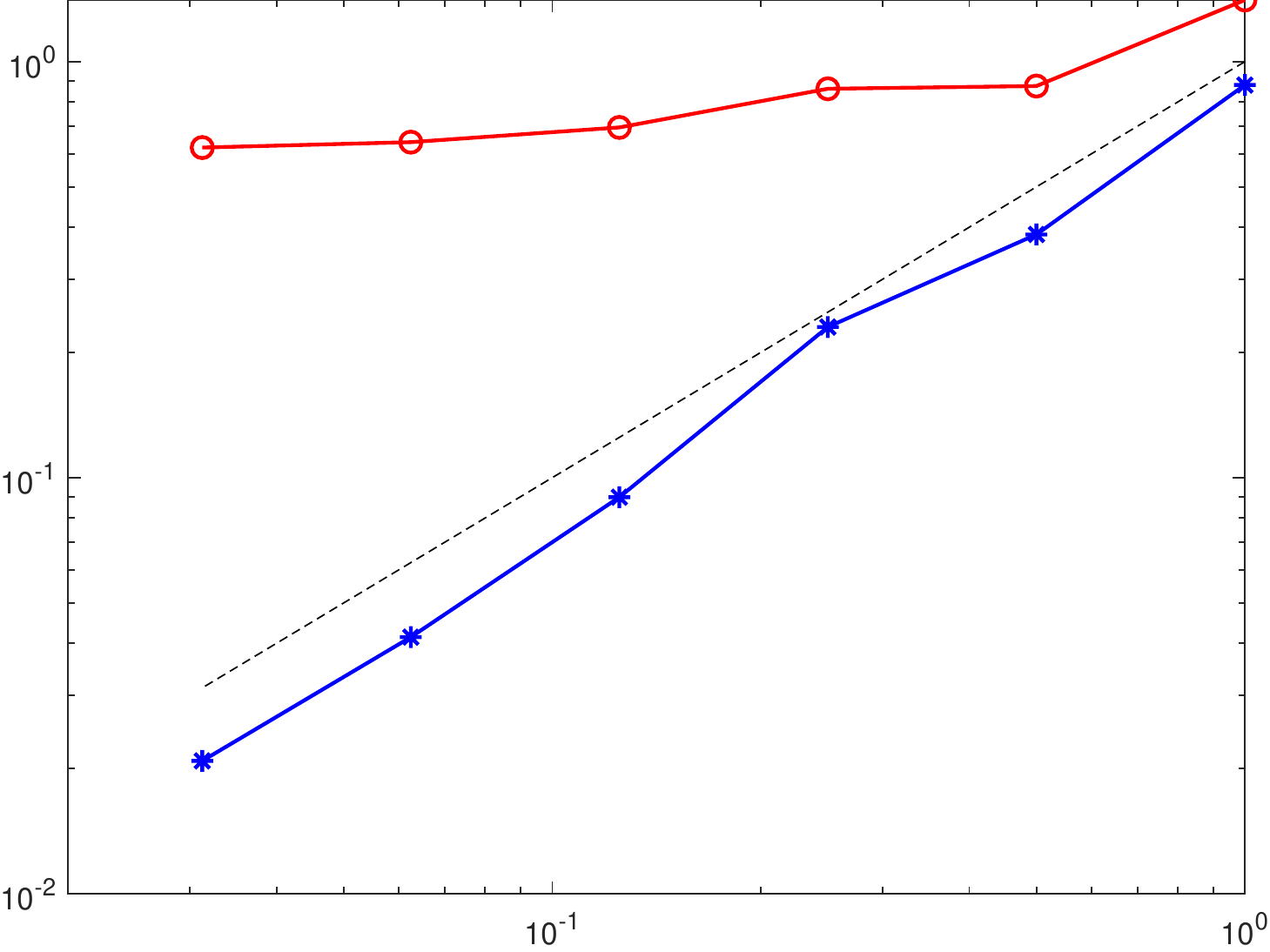}
		\caption{Two dimensions}
	\end{subfigure} ~
	\begin{subfigure}{0.45\textwidth}
		\includegraphics[width=\textwidth]{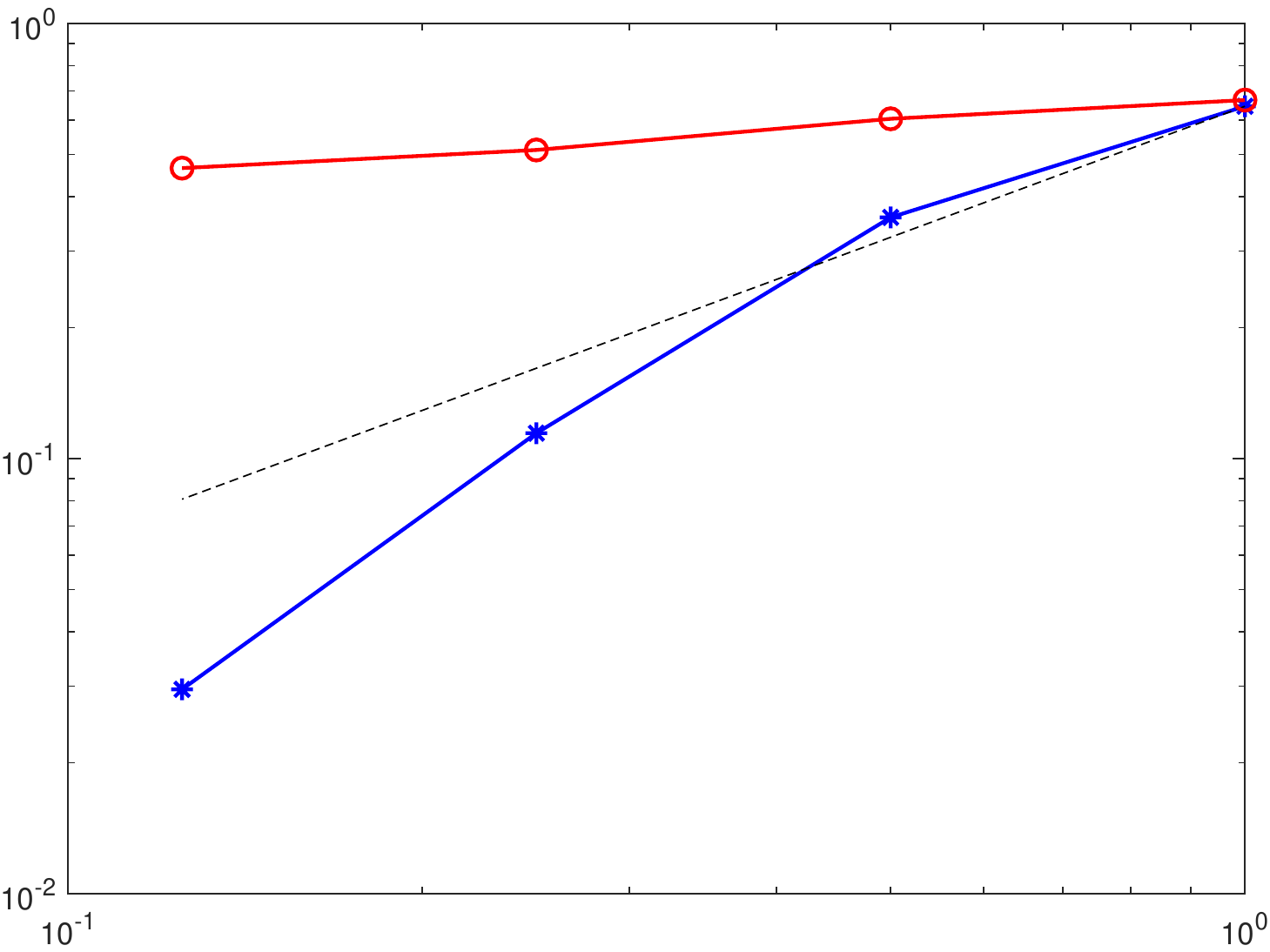}
		\caption{Three dimensions}
	\end{subfigure}
	\caption{Relative errors for the LOD (blue $\ast$) and the FEM (red $\circ$) approximations of Example 1 plotted against the mesh size $H$. The dashed line is $cH$.}\label{fig:Example1_2D}
\end{figure}

\subsection{Example 2 - Natural boundary conditions}\label{sec:example2}
Here we show that the convergence is of order $H^{1/2}$ if $\mathbf{f}\cdot \mathbf{n} \neq 0$, in accordance with Theorem~\ref{theorem-local-LOD}.

In $2D$ we choose $\mathbf{f}(x,y) = [1, 1]$. The meshes for the LOD approximations and the reference mesh remain the same as in Example~\ref{sec:example1}. 

First, to demonstrate that the reduced convergence does not depend on the localization parameter $m$, we use the ideal method. That is, we do not use any localization when computing the LOD approximations. In Figure~\ref{fig:Example2_2D} we clearly see that the $H^{1/2}$-term dominates the convergence ($H^{1/2}$ marked by a dashed line). In the plot the error for iteration $j=0$ is omitted for the LOD approximation, because the error in this particular setting is close to zero. We also note that, despite the reduced convergence rate, the method still outperforms the classical FEM in terms of accuracy relative to the number of degrees of freedom in the solution space. 

To improve the error we test the source term corrections proposed in Section~\ref{subsection:source-corrections}. We present results in $2D$ for two cases: 1. when the correction is computed for all triangles and 2. when it is only computed for triangles that intersect the outer boundary $\Gamma$. In Table~\ref{correction_table} we see how this affects the error for some values of $H$ and $m$. As expected, computing the corrections for all triangles in the domain decreases the error drastically. Note that for the ideal method (without localization) the error would vanish completely. Furthermore, we can see that computing the corrections only for triangles intersecting the boundary still reduces the error significantly. 

Essentially, source corrections allow for an arbitrary accuracy depending on the size of the localization patches (i.e., the value of $m$) and the number of coarse elements on which the source correction is computed. From our experiments we see that one layer of coarse elements around the boundary is often sufficient to obtain a reasonable accuracy.

\begin{figure}[h]
	\centering
	\includegraphics[width=0.5\textwidth]{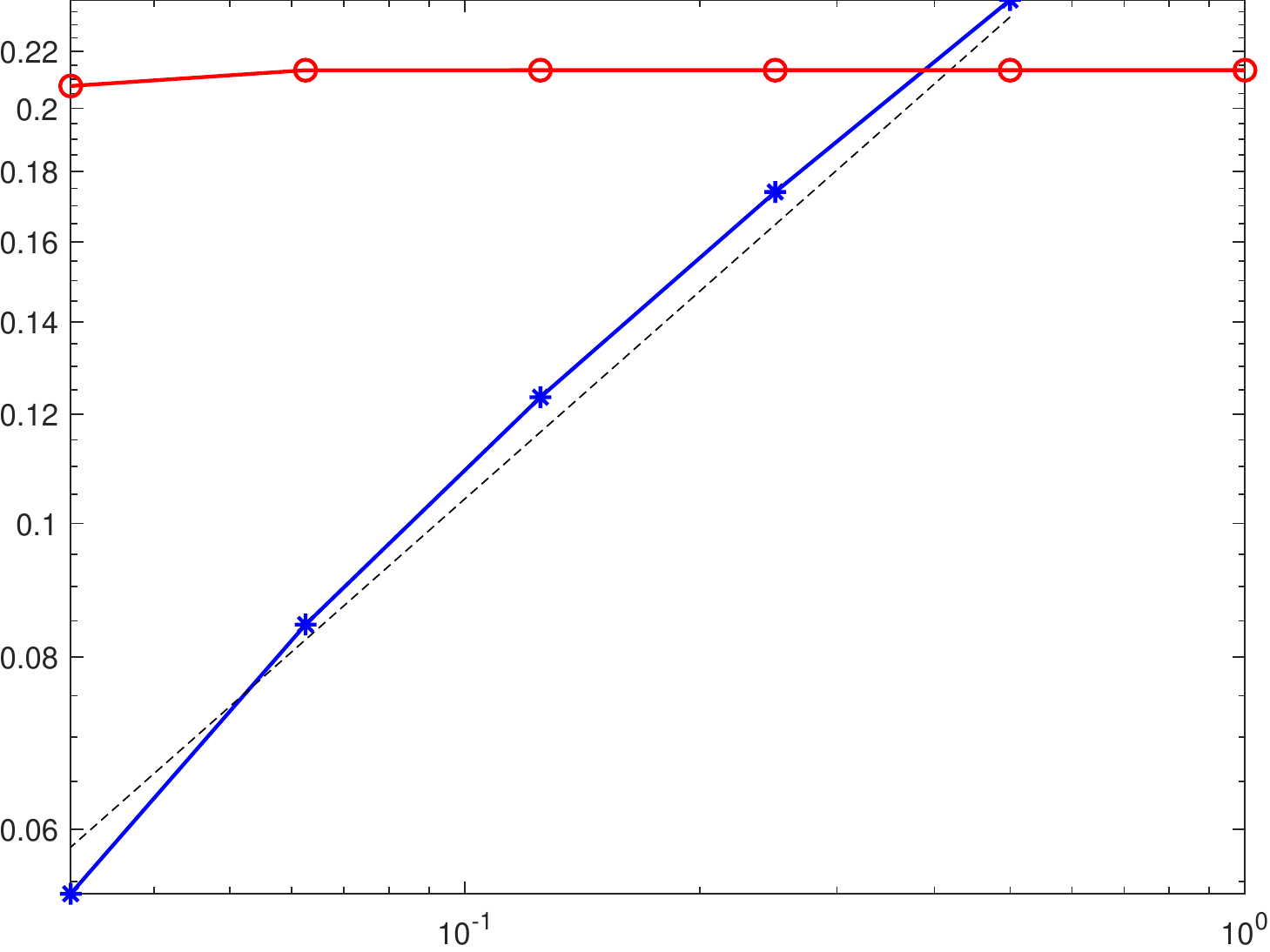}
	\caption{Relative errors for the LOD without localization (blue $\ast$) and the FEM (red $\circ$) approximations of Example 2 plotted against the mesh size $H$. The dashed line is $cH^{1/2}$.}\label{fig:Example2_2D}
\end{figure}

\begin{table}[h]
	\caption{Relative errors for the LOD approximation of Example 2 with and without correction for the source term.}\label{correction_table}
	\begin{tabular}{|l|c|c|}
		\hline
		& $H=\sqrt{2}\cdot 2^{-2}$, $m=2$  & $H=\sqrt{2}\cdot 2^{-3}$, $m=3$ \\ \hline
		LOD without correction & 0.1731 & 0.1235 \\ \hline
		Corrections for triangles at the boundary & 0.101  &  0.0828 \\ \hline
		Corrections for all triangles & $0.738\cdot 10^{-4}$ & $0.261\cdot 10^{-4}$ \\ \hline
	\end{tabular}
\end{table}


\subsection{Example 3 - Essential boundary conditions}\label{sec:example3}
In this example we test the convergence rate for the problem with essential boundary conditions, i.e., problem \eqref{curlcurl_eq_essential}, as described in Section~\ref{sec:essential_bc} when $\mathbf{f} \cdot \mathbf{n}=0$ is fulfilled.
We choose the same setting as in Section~\ref{sec:example1}, except that we impose homogeneous essential boundary conditions instead. In Figure~\ref{fig:Example3_2D+3D}, we clearly see that linear order convergence is achieved in $2D$. In $3D$ linear convergence is achieved except for the coarsest mesh size. The dashed line is plotted from the second point. The experiment demonstrates that if $\mathbf{f} \cdot \mathbf{n}=0$, then the same Falk-Winther projection can be used for solving both the curl-curl-problem with essential boundary conditions and the curl-curl-problem with natural boundary conditions. In both cases, an optimal linear convergence rate is obtained.

\begin{figure}[h]
	\centering
	\begin{subfigure}{0.45\textwidth}
		\includegraphics[width=\textwidth]{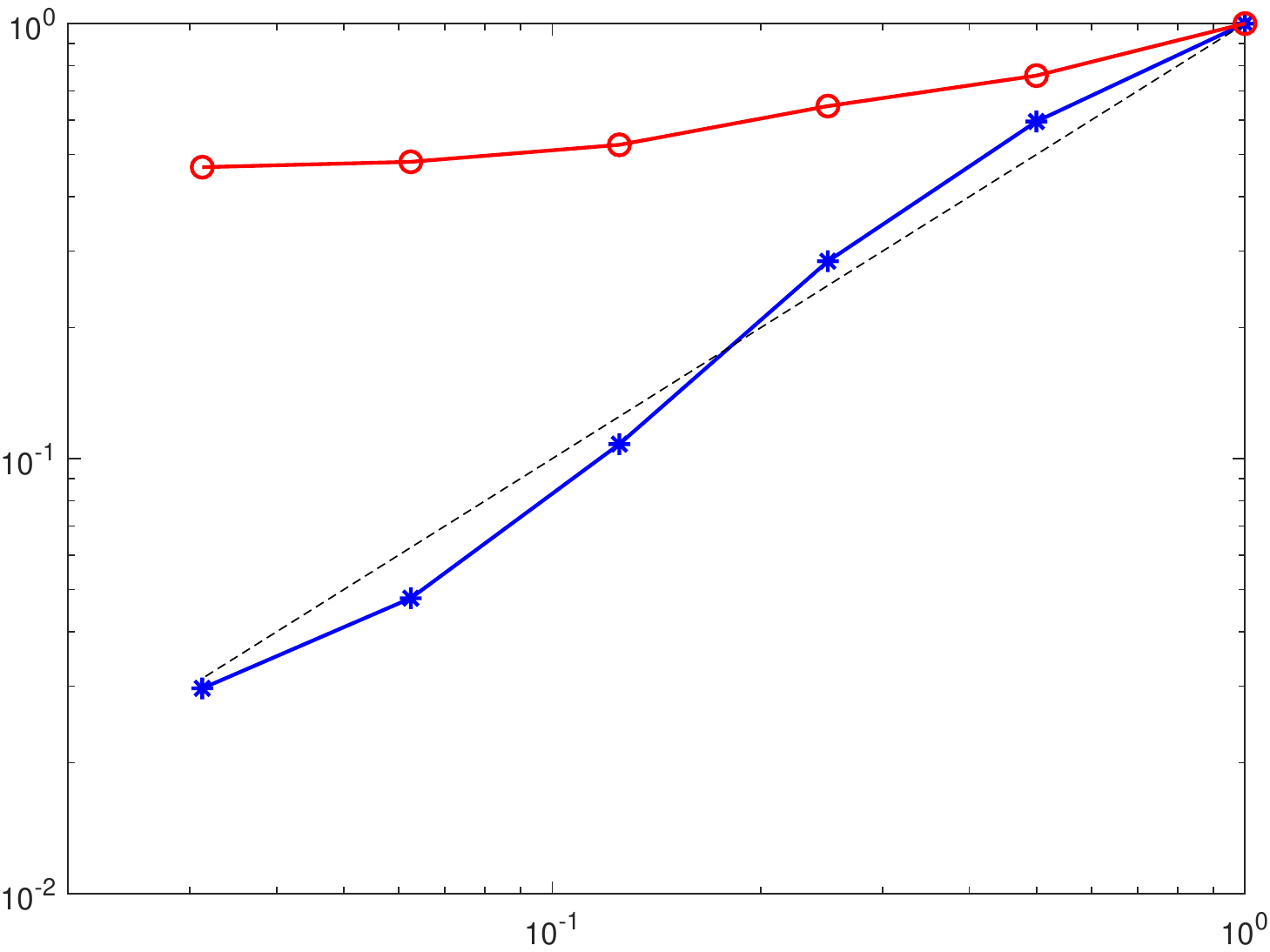}
		\caption{Two dimensions}
	\end{subfigure} ~
	\begin{subfigure}{0.45\textwidth}
		\includegraphics[width=\textwidth]{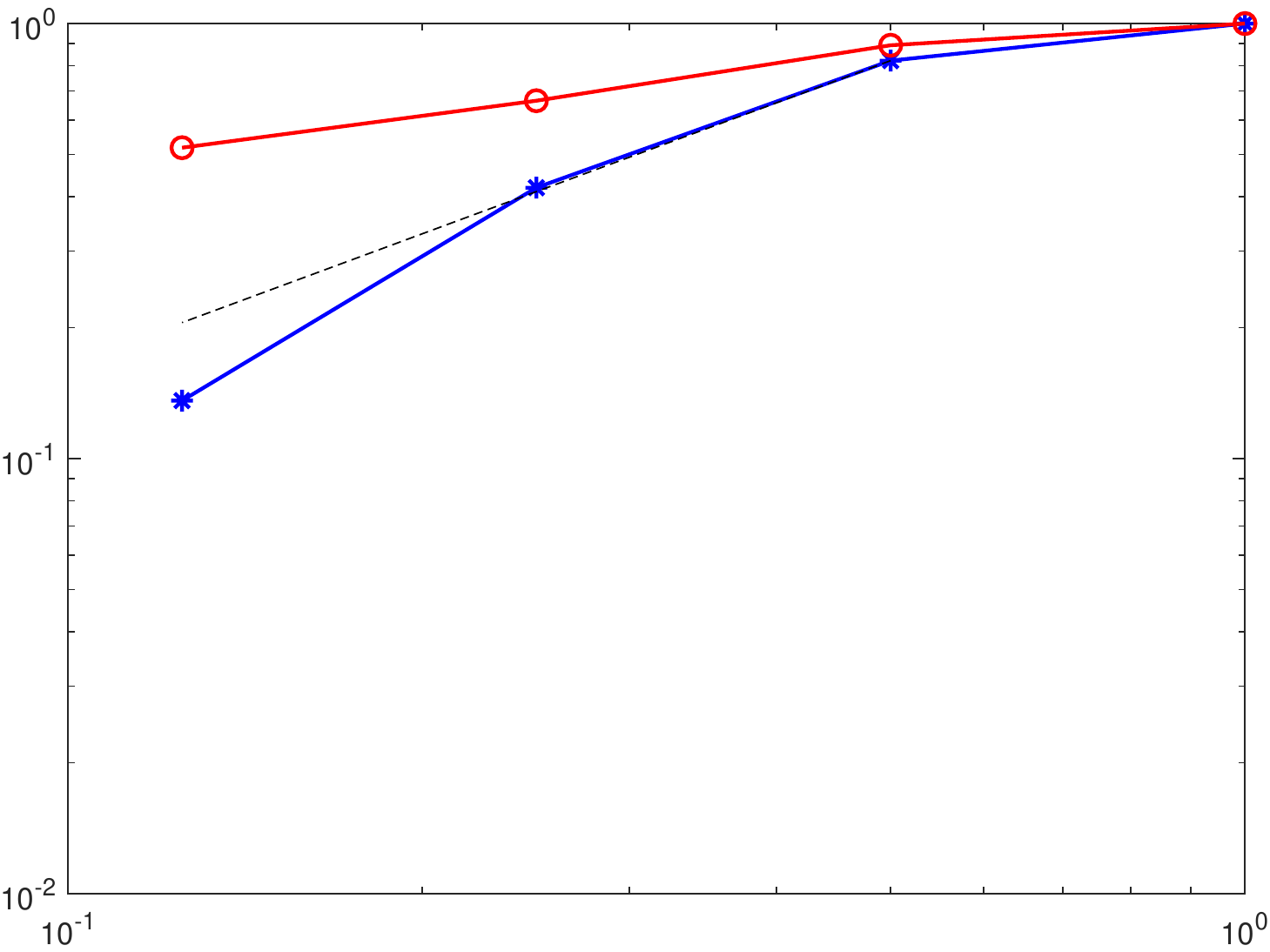}
		\caption{Three dimensions}
	\end{subfigure}
	\caption{Relative errors for the LOD (blue $\ast$) and the FEM (red $\circ$) approximations of Example 3 plotted against the mesh size $H$. The dashed line is $cH$.}\label{fig:Example3_2D+3D}
\end{figure}

\subsection{Example 4 - Essential boundary conditions}\label{sec:example4}
In this example we consider again essential boundary conditions, i.e., problem  \eqref{curlcurl_eq_essential}. Otherwise we put ourselves in the same setting as in Section~\ref{sec:example2}.

First, we conclude that the convergence rate is of order $H^{1/2}$ if $\mathbf{f}\cdot \mathbf{n} \neq 0$, see Figure~\ref{fig:Example4_2D}. Here we have used the LOD method without localization to show that the reduced convergence rate is not due to the localization parameter. We also note that the method still outperforms the classical FEM (in terms of accuracy relative to the number of degrees of freedom), despite the reduced convergence rate. 

As in Section~\ref{sec:example2} we also try to improve the error by using the source term correctors. The results are displayed in Table~\ref{correction_table_2} for some values of $H$ and $m$. As we can see, the error is improved drastically when computing the correction for all triangles. If the computations are restricted to the triangles intersecting the boundary, the error is still improved.

\begin{figure}[h]
	\centering
	\includegraphics[width=0.5\textwidth]{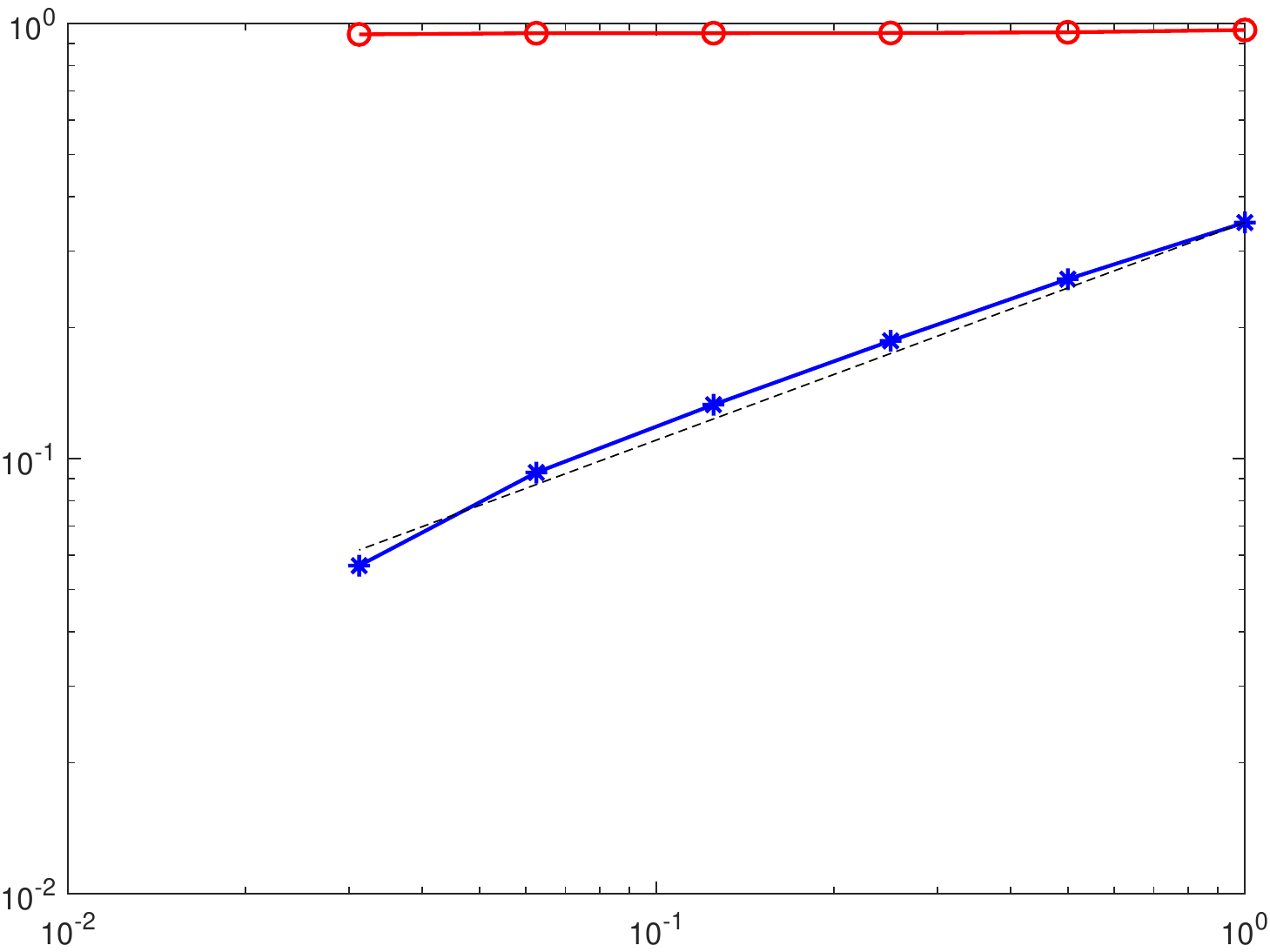}
	\caption{Relative errors for the LOD without localization (blue $\ast$) and the FEM (red $\circ$) approximations of Example 4 plotted against the mesh size $H$. The dashed line is $cH^{1/2}$.}\label{fig:Example4_2D}
\end{figure}

\begin{table}[h]
	\caption{Relative errors for the LOD approximation of Example 4 with and without correction for the source term.}\label{correction_table_2}
	\begin{tabular}{|l|c|c|}
		\hline
		& $H=\sqrt{2}\cdot 2^{-2}$, $m=2$  & $H=\sqrt{2}\cdot 2^{-3}$, $m=3$ \\ \hline
		LOD without correction & 0.186 & 0.134 \\ \hline
		Corrections for triangles at the boundary & 0.117 & 0.0964 \\ \hline
		Corrections for all triangles & $3.92\cdot 10^{-3}$ & $2.78 \cdot 10^{-3}$ \\ \hline
	\end{tabular}
\end{table}

Finally, we comment on the operator $\pi^{0,E}_H = \mathcal{I}^{0,E}_H \circ \pi^E_H$, see Section~\ref{sec:essential_bc}. This is the operator which puts all weights to zero for the edge-basis functions $\boldsymbol{\psi}_E$ corresponding to a boundary edge. In Figure~\ref{fig:Example4_bc}, we see that this approach still shows reduced convergence of order $\sqrt{H}$, which is similar to the current approach of just keeping $\pi^H_E$. We conclude that it is not enough to simply put all weights to zero to achieve the Falk-Winther projection with essential boundary conditions $\mathring\pi^E_H$ that is required for the results of \cite{Gallistl18} to hold true (i.e., a linear order convergence in $H$ independent of the normal trace of $\f$).

\begin{figure}[h]
	\centering
	\includegraphics[width=0.5\textwidth]{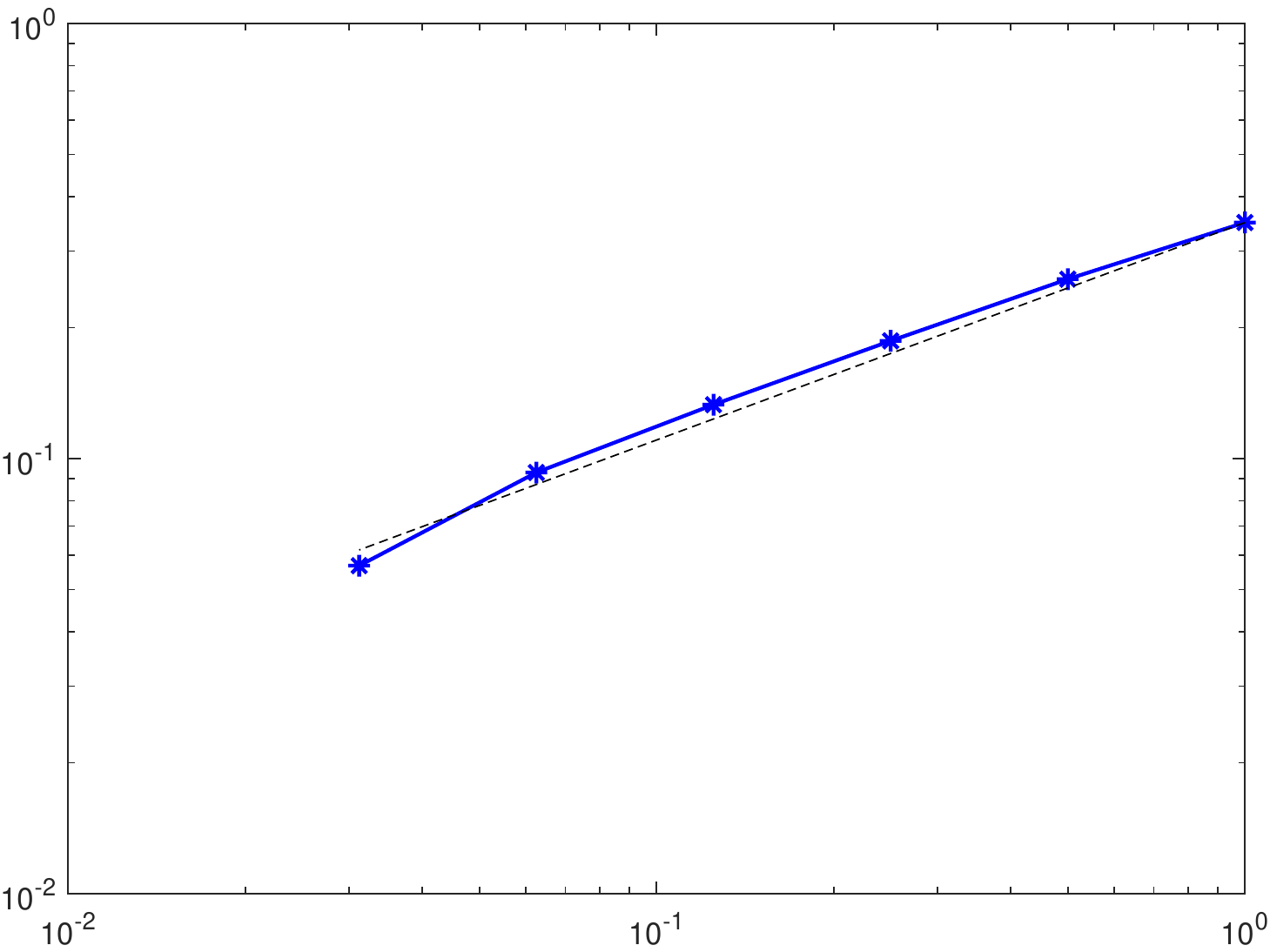}
	\caption{Relative errors for the LOD approximation of Example 4 without localization (blue $\ast$) based on $\pi^{0,E}_H$ plotted against the mesh size $H$. The dashed line is $cH^{1/2}$.}\label{fig:Example4_bc}
\end{figure}

\medskip
$\\$
{\bf Acknowledgements.}
The authors thank Barbara Verf\"urth for the many fruitful discussions on the general methodology and Joachim Sch\"oberl for the helpful and valuable comments on the Falk-Winther operator.

\appendix

\section{Construction of the Falk-Winther operator}
\label{section-falk-winther-construction}

In the following we describe the implementation of the Falk-Winther projection operator $\pi_H^E: \Hcurl \rightarrow \mathcal{N}(\TauH)$ as constructed in \cite{FalkWinther2015,FalkWinther2014}. Here we note that a discussion of its implementation can be already partially found in \cite{Gallistl18}, which however contains some minor mistakes, which are fixed in the following descriptions. Our implementation is discussed for the $3D$ case, but we also comment on what needs to be changed for the $2D$ case. 

Adapting the notation from the aforementioned references, we let $\Delta_0(\TauH)$ denote the set of vertices of $\mathcal{T}_H$ and $\Delta_1(\TauH)$ is the set of edges. The space $\mathcal{N}(\mathcal{T}_H)$ is spanned by the edge-oriented
basis functions $(\boldsymbol{\psi}_E)_{E\in\Delta_1(\TauH)}$ that are uniquely defined for any $E\in \Delta_1(\TauH)$
through the property
\begin{equation*}
\int_E \boldsymbol{\psi}_E\cdot \mathbf{t}_E\,ds = 1
\quad\text{and}\quad
\int_{E'} \boldsymbol{\psi}_E\cdot \mathbf{t}_E\,ds = 0
\quad\text{for all }E'\in\Delta_1(\TauH)\setminus\{E\}.
\end{equation*}
Here, $\mathbf{t}_E$ denotes the unit tangent to the edge $E$ with a globally
fixed sign. In the following we denote by $y_1(E)$ and $y_2(E)$ the endpoints of $E$, where we make the orientation convention that
$$
\mathbf{t}_E = (y_1(E)-y_2(E))/|E|, \qquad \mbox{where } |E|:=\operatorname{length}(E).
$$
The (edge-based) Falk-Winter projection $\pi_H^E:\Hcurl \to \mathcal{N}(\mathcal{T}_H)$ is of the form
\begin{equation}\label{e:R1def}
\pi_H^E \mathbf{u}
:=
S^1 \mathbf{u}
+
\sum_{E\in\Delta_1(\TauH)}
\int_E \big((\id -S^1)Q^1_E \mathbf{u}\big)\cdot \mathbf{t}_E\,ds
\,\boldsymbol{\psi}_E .
\end{equation}
Here, $S^1$ is constructed in such a way that $\pi_H^E$ commutes with the nodal Falk-Winther projection (i.e., $ \pi_H^E( \nabla v ) =  \nabla \pi_H^V( v )$ for all $v \in H^1(\Omega)$). However, since $S^1$ is not a projection, the second term involving the operator $Q^1_E$ needs to be introduced. This ensures that $\pi_H^E$ is invariant on $\mathcal{N}(\mathcal{T}_H)$. We shall describe the construction of the various constituents in the following subsections, where will also comment on their practical realization.

\subsection{Construction of $\pi_H^V$}
For a better understanding of the construction, we start with describing the nodal Falk-Winther projection $\pi_H^V : H^1(\Omega) \rightarrow \HLag$. For that, we denote for each vertex $y\in\Delta_0(\TauH)$ the associated nodal patch by
\begin{equation*}
\omega_y:=\Int\Big(\bigcup\{T\in\mathcal{T}_H : y\in T\}\Big).
\end{equation*}
With this, we also define the piecewise constant function $z^0_y$ by
\begin{equation}\label{def-z0y}
z^0_y = \begin{cases}
(\operatorname{meas}(\omega_{y}))^{-1} &\text{in } \omega_{y} \\
0                      &\text{in } \Omega\setminus\omega_{y}
\end{cases}
\end{equation}
Furthermore, the restriction of $\mathcal{S}(\mathcal{T}_H)$ to $\omega_y$ is denoted by  $\mathcal{S}(\mathcal{T}_H(\omega_y))$ (i.e., scalar-valued first-order Lagrange finite elements on $(\mathcal{T}_H)_{\vert \omega_y}$).

The construction of $\pi_H^V$ is simply based on a local $H^1$-projection on each patch $\omega_y$. More precisely, we let $Q_y^0 : H^1(\omega_y) \rightarrow \mathcal{S}(\mathcal{T}_H(\omega_y))$ be defined through the solution to a local Neumann problem, where $Q_y^0(v) \in \mathcal{S}(\mathcal{T}_H(\omega_y))$ solves
\begin{align*}
\int_{\omega_y} \nabla Q_y^0(v) \cdot \nabla w = \int_{\omega_y} \nabla v \cdot \nabla w \qquad \mbox{for all } 
w \in \mathcal{S}(\mathcal{T}_H(\omega_y))
\end{align*}
under the constraint $\int_{\omega_y}  Q_y^0(v) \hspace{2pt} dx = 0$. Since $Q_y^0$ does not preserve constants, we need to restore them through a local averaging operator $M^0 : L^2(\Omega) \rightarrow \HLag$, which is defined by
\begin{align*}
M^0 v := \sum_{y \in \triangle_0(\TauH)}  \left( \int_{\omega_y} z^0_y \hspace{2pt} v \hspace{2pt} dx \right) \lambda_y.
\end{align*}
Here, $\lambda_y \in \HLag$ is the standard nodal basis function (hat function) associated with the vertex $y$, i.e., it fulfills $\lambda_y(\tilde{y}) = \delta_{y \tilde{y}}$ for all $\tilde{y} \in \Delta_0(\TauH)$.

With this, the node-based Falk-Winther projection is given by
\begin{align*}
\pi_H^V(v) := M^0(v) + \sum_{y \in \triangle_0(\T_H)} (\hspace{1pt}Q_y^0(v)\hspace{1pt})(y) \hspace{3pt} \lambda_y.
\end{align*}
The operator is designed in the same way in $2D$. Though the construction of $\pi_H^V$ is not required for the construction of $\pi_H^E$, it can be helpful for validating the implementation through checking the commuting diagram property.

\subsection{Construction of $S^1$}
\label{subsection-construction-S1}
In the next step, we discuss the construction of $S^1$ in \eqref{e:R1def}. We start with defining the extended edge patch, for an edge $E\in \Delta_1(\TauH)$, as the union of the two nodal patches that are associated with the corners of the edge. More precisely, for an edge $E=\operatorname{conv}\{y_1,y_2\}\in\Delta_1(\TauH)$ with vertices (corners) $y_1,y_2\in\Delta_0(\TauH)$, the extended edge patch is given by
\begin{equation*}
\omega_E^{\mathit{ext}} := \omega_{y_1}\cup\omega_{y_2}. 
\end{equation*}
Furthermore, recalling the definition of $z_0^y$ from \eqref{def-z0y}, we define for any edge $E=\operatorname{conv}\{y_1,y_2\}\in\Delta_1(\TauH)$ with vertices
$y_1,y_2\in \Delta_0(\TauH)$ the piecewise constant function $(\delta z^0)_E \in L^2(\Omega)$ by
\begin{equation*}
(\delta z^0)_E := z^0_{y_1} - z^0_{y_2}.
\end{equation*}
Before we can proceed, we need some additional notation. For that, we let $\mathcal{T}_H(\omega_E^{\mathit{ext}})$ be the restriction of the mesh $\mathcal{T}_H$ to the extended edge-patch $\omega_E^{\mathit{ext}}$. We define the space of lowest-order N\'ed\'elec finite elements over $\mathcal{T}_H(\omega_E^{\mathit{ext}})$ and with a zero tangential trace on $\partial \omega_E^{\mathit{ext}}$ by
\begin{align*}
\mathring{\mathcal{N}}(\mathcal{T}_H(\omega_E^{\mathit{ext}})):=
\{ \bfv_H \in \mathcal{N}(\mathcal{T}_H(\omega_E^{\mathit{ext}})) | \hspace{3pt}
\bfv_H \times \mathbf{n} = 0 \mbox{ on } \partial \omega_E^{\mathit{ext}} \}.
\end{align*}
Similarly, we let $\mathring{\mathcal{RT}}(\mathcal{T}_H(\omega_E^{\mathit{ext}}))$ denote the space of lowest-order Raviart--Thomas finite elements over $\mathcal{T}_H(\omega_E^{\mathit{ext}})$ and with a zero normal trace on $\partial \omega_E^{\mathit{ext}}$, i.e., 
\begin{align*}
\mathring{\mathcal{RT}}(\mathcal{T}_H(\omega_E^{\mathit{ext}})):=
\{ \bfv_H \in \mathcal{RT}(\mathcal{T}_H(\omega_E^{\mathit{ext}})) | \hspace{3pt}
\bfv_H \cdot \mathbf{n} = 0 \mbox{ on } \partial \omega_E^{\mathit{ext}} \}.
\end{align*}
Given $E\in \Delta_1(\TauH)$, we search for a representation of $(\delta z^0)_{E}$ as the divergence of a $\mathring{\mathcal{RT}}$-function. More precisely, we seek for each $(\delta z^0)_{E}$ a function
$$
\mathbf{z}_E^1\in \mathring{\mathcal{RT}}(\mathcal{T}_H(\omega_E^{\mathit{ext}}))
$$
such that
\begin{align*}
\Div \mathbf{z}_E^1 =-(\delta z^0)_E  \qquad
\mbox{and} \qquad
(\mathbf{z}_E^1,\curl \boldsymbol{\tau}) &= 0 \quad \mbox{for all } \boldsymbol{\tau}\in \mathring{\mathcal{N}}(\mathcal{T}_H(\omega_E^{\mathit{ext}})).
\end{align*}
This problem is well-posed, since $(\delta z^0)_{E}$ is $\TauH$-piecewise constant with zero average over $\omega_E^{\mathit{ext}}$. Existence and uniqueness then follow from the exactness of the corresponding discrete complex with vanishing traces (cf. \cite{FalkWinther2014}). A motivation for the crucial role of $\mathbf{z}_E^1$ for the commuting property of the projection $\pi_H^E$ is given in Section \ref{subsection:commutation-S1-piE} below.

\begin{remark}
	In practice, $\mathbf{z}^1_E$ can be found by solving the following saddle point problem. Find $(\mathbf{z}^1_E, \mathbf{v}) \in \mathring{\mathcal{RT}}(\mathcal{T}_H(\omega_E^{\mathit{ext}})) \times \mathring{\mathcal{N}}(\mathcal{T}_H(\omega_E^{\mathit{ext}}))$ such that
	\begin{equation*}
	\begin{aligned}
	(\Div \mathbf{z}^1_E, \Div \mathbf{v}) + (\mathbf{w}, \curl \mathbf{v}) &= (-(\delta z^0)_E, \Div \mathbf{w})
	\quad&&\text{for all } \mathbf{w}\in \mathring{\mathcal{RT}}(\mathcal{T}_H(\omega_E^{\mathit{ext}}))
	\\
	(\mathbf{z}^1_E, \curl \boldsymbol{\tau}) & = 0 && \text{for all } \boldsymbol{\tau}\in \mathring{\mathcal{N}}(\mathcal{T}_H(\omega_E^{\mathit{ext}})).
	\end{aligned}
	\end{equation*}
\end{remark}

Using the functions $\mathbf{z}_E^1$, we define the operator $M^1:L^2(\Omega;\C^3)\to\mathcal{N}(\mathcal{T}_H)$ that maps any function $\mathbf{u}\in L^2(\Omega;\C^3)$ to
\begin{equation*}
M^1\mathbf{u} :=
\sum_{E\in\Delta_1(\TauH)}
\int_{\omega_E^{\mathit{ext}}} \mathbf{u}\cdot \mathbf{z}_E^1\,dx\, \boldsymbol{\psi}_E.
\end{equation*}

As a last constituent of $S^1$, we define the operator 
$$
Q^1_{y,-} : \mathbf{H}(\curl,\omega_y)
\to 
\mathcal{S}(\mathcal{T}_H(\omega_y))
$$ 
again through the solution of a local discrete Neumann problem. More precisely, for $\mathbf{u}\in \mathbf{H}(\curl, \omega_y )$, the image $ Q^1_{y,-} \mathbf{u} \in \mathcal{S}(\mathcal{T}_H(\omega_y))$ is the solution to
\begin{equation*}
\begin{aligned}
(\mathbf{u}-\nabla Q^1_{y,-} \mathbf{u},\nabla v) &= 0
\quad&&\text{for all } v\in \mathcal{S}(\mathcal{T}_H(\omega_y))
\\
\int_{\omega_y} Q^1_{y,-} \mathbf{u}\,dx & = 0. &&
\end{aligned}
\end{equation*}
Note the close relation of $Q^1_{y,-}$ to the operator $Q^0_y$ from Section \ref{subsection-construction-S1}.

With all the parts available, we can define the operator 
$S^1:\mathbf{H}_0(\curl,\Omega)\to \mathring{\mathcal{N}}(\mathcal{T}_H)$
via
\begin{equation}\label{e:S1def1}
S^1 \mathbf{u} :=
M^1 \mathbf{u} +
\sum_{y\in\Delta_0(\TauH)} (Q^1_{y,-}\mathbf{u})(y)\nabla \lambda_y .
\end{equation}
From a practical point of view, it is useful to rewrite the operator $S^1$ in terms of the edge-based basis functions $\boldsymbol{\psi}_E$ of $\mathcal{N}(\TauH)$. Given a vertex 
$y\in\Delta_0(\TauH)$, we can expand $\lambda_y$ in terms of the basis functions
$(\boldsymbol{\psi}_E)_{E\in\Delta_1(\mathcal{T}_H)}$. We obtain
\begin{equation*}
\nabla\lambda_z
=
\sum_{E\in\Delta_1(\mathcal{T}_H)} \int_E \nabla\lambda_z\cdot \mathbf{t}_E\,ds\,\boldsymbol{\psi}_E
=
\sum_{E\in\Delta_1(z)} \text{sign}(\mathbf{t}_E \cdot \nabla \lambda_z) \hspace{2pt} \boldsymbol{\psi}_E
\end{equation*}
where $\Delta_1(z)\subseteq \Delta_1(\mathcal{T}_H)$
is the set of all edges that contain $z$. Thus, we can rewrite the operator $S^1$ from \eqref{e:S1def1} as
\begin{equation}\label{e:S1def2}
S^1 \mathbf{u} :=
M^1 \mathbf{u} +
\sum_{E\in\Delta_1(\TauH)}
\big((Q^1_{y_1(E),-}\mathbf{u})(y_1(E)) - (Q^1_{y_2(E),-}\mathbf{u})(y_2(E))\big) \boldsymbol{\psi}_E.
\end{equation}
At this point we have the commuting property, but $S^1$ is not a projection yet.

\begin{remark}\label{S_assemble}
	For the practical realization of the operator $S^1$, we seek a matrix representation, $S$, of $S^1: \mathcal{N}(\mathcal{T}_h) \rightarrow \mathcal{N}(\TauH)$, where $\mathcal{T}_h$ is a refinement of $\TauH$ that resolves the multiscale variations of the problem.
	For a given edge $E \in \Delta_1(\TauH)$, define the vector $M_E$
	\begin{equation*}
	M_E := \left[\int_{\omega_E^{\mathit{ext}}} \mathbf{z}^1_E \cdot \boldsymbol{\psi}_{e} \,dx\,\right]_{e \in \mathcal{T}_h(\omega_E^{\mathit{ext}})}.
	\end{equation*}
	To compute $Q^1_{y,-}$, the mean constraint is enforced by a classical Lagrange multiplier, which leads to a saddle point system. For a node $y \in \Delta_0(\TauH)$ define the matrices 
	\begin{equation*}
	A_y := \left[\int_{\omega_y} \nabla \lambda_z \cdot \nabla \lambda_w \,dx\, \right]_{z,w \in \Delta_0(\TauH(\omega_y))} \quad B_y:= \left[ \int_{\omega_y} \nabla \lambda_z \cdot \boldsymbol{\psi}_e  \,dx\,\right]_{\begin{subarray}{l} z\in \Delta_0(\TauH(\omega_y)) \\e \in \Delta_1(\mathcal{T}_h(\omega_y)) \end{subarray}}
	\end{equation*}
	and the vector
	\begin{equation*}
	C_y = \left[\int_{\omega_y} \lambda_z \,dx\,\right]_{z \in \Delta_0(\TauH(\omega_y))}.
	\end{equation*} 
	Solve the following system for $Q_y$
	\begin{equation*}
	\begin{bmatrix}
	A_y & C^\ast_y \\ C_y & 0
	\end{bmatrix}
	\begin{bmatrix}
	Q_y \\ V_y
	\end{bmatrix}
	=
	\begin{bmatrix}
	B_y \\ 0
	\end{bmatrix}
	\end{equation*}
	and let $Q_j = Q_y\left[y_j(E),:\right]$, for $j=1,2$. Finally, the matrix representation of $S^1$ can be expressed as
	\begin{equation*}
	S(E,[e_1,...,e_N]) = M_E + Q_1 - Q_2,
	\end{equation*}
	where $e_1,....,e_N$ denotes the edges in the patch $\mathcal{T}_h(\omega_E^{\mathit{ext}})$.
\end{remark}

\begin{remark}
	In the $2D$ case, the computation of $z^1_E$ needs to be slightly changed. The lowest-order N\'ed\'elec space with zero tangential trace on $\partial \omega_E^{\mathit{ext}}$, $\mathring{\mathcal{N}}(\mathcal{T}_H(\omega_E^{\mathit{ext}}))$, should be replaced with the first-order (scalar-valued) Lagrange finite element space with zero trace on $\partial \omega_E^{\mathit{ext}}$, that is,
	\begin{equation*}
	\mathring{\mathcal{S}}(\mathcal{T}_H(\omega_E^{\mathit{ext}})) := \{v \in \mathcal{S}(\TauH)| \hspace{3pt} v = 0 \mbox{ on } \partial \omega_E^{\mathit{ext}}\}.
	\end{equation*} 
	Except for $\curl$ having a different meaning in $2D$, this is the only change in the interpolation $\pi^E_H$ that needs to be made for the $2D$ case.
\end{remark}

\subsection{Commuting property $\nabla \circ \pi_H^V = S^1 \circ \nabla$}
\label{subsection:commutation-S1-piE}

For a better understanding of the construction of $S^1$, we will demonstrate that it indeed implies the commuting property
$$\nabla \pi_H^V(v) = S^1(\nabla v) \qquad \mbox{for all } v \in H^1(\Omega).
$$
We start with applying $S^1$ to a gradient $\nabla v$ for some $v \in H^1(\Omega)$. First, we easily observe
$$
Q^1_{y,-}(\nabla v) = Q_y^0(v).
$$
Next, for an edge $E=\{y_1,y_2\}$ we consider the tangential component of $\nabla M^0(v)$ (i.e., the local degree of freedom) which is given by
\begin{eqnarray*}
	\lefteqn{\int_E   \nabla M^0(v) \cdot \mathbf{t}_E\,ds = |E|^{-1} \int_{E} \nabla M^0(v) \cdot (y_1 - y_2)} \\
	&=&
	|E|^{-1} \int_{E} \left( \left( \int_{\omega_{y_1}} z^0_{y_1} \hspace{2pt} v \hspace{2pt} dx \right) \nabla \lambda_{y_1}
	+ \left( \int_{\omega_{y_2}} z^0_{y_2} \hspace{2pt} v \hspace{2pt} dx \right) \nabla \lambda_{y_2}
	\right) \cdot (y_1 - y_2) \\
	&=& |E|^{-1} \int_{E}  \left( \int_{\omega_{y_1}} z^0_{y_1} \hspace{2pt} v \hspace{2pt} dx -  \int_{\omega_{y_2}} z^0_{y_2} \hspace{2pt} v \hspace{2pt} dx \right) \nabla \lambda_{y_1}
	\cdot (y_1 - y_2) \\
	&=& \left( \int_{\omega_{y_1}} z^0_{y_1} \hspace{2pt} v \hspace{2pt} dx -  \int_{\omega_{y_2}} z^0_{y_2} \hspace{2pt} v \hspace{2pt} dx \right) \\
	&=&  \int_{\omega_E^{\mathit{ext}}} (\delta_z^0)_E \hspace{2pt} v \hspace{2pt} dx = - \int_{\omega_E^{\mathit{ext}}} \Div \mathbf{z}_E^1   \hspace{3pt} v \hspace{2pt} dx = \int_{\omega_E^{\mathit{ext}}}  \mathbf{z}_E^1   \hspace{2pt} \nabla v \hspace{2pt} dx.
\end{eqnarray*}
Since $\nabla M^0(v) \in \mathcal{N}(\mathcal{T}_H)$ is uniquely determined by these degrees of freedom, we have
\begin{align*}
\nabla M^0(v) = \sum_{E \in \Delta_1(\TauH)} \left( \int_{\omega_E^{\mathit{ext}}}  \mathbf{z}_E^1   \hspace{2pt} \nabla v \hspace{2pt} dx \right) \psi_{E} = M^1(\nabla v).
\end{align*}
Combining the equations $Q^1_{y,-}(\nabla v) = Q_z^0(v)$ and $\nabla M^0(v) =M^1(\nabla v)$ for all $v\in H^1(\Omega)$, we obtain
\begin{align*}
\nabla \pi_H^V(v) = \nabla M^0(v) + \sum_{z \in \mathring{\triangle}_0(\T_H)} (Q_z^0 v)(z) \hspace{3pt} \nabla \lambda_z
= M^1(\nabla v) 
+ \sum_{y \in \mathring{\triangle}_0(\T_H)} Q^1_{y,-}(\nabla v)(y) \hspace{3pt} \nabla \lambda_y
\overset{\eqref{e:S1def1}}{=} S^1(\nabla v)
\end{align*}
as desired. Observe that it can be helpful for the validation of the Falk-Winther implementation to check the commuting property $\nabla \pi_H^V v = S^1(\nabla v)$ numerically.

\subsection{Construction of $Q^1_E$}
As mentioned before, the operator $S^1$ commutes with vertex-based Falk-Winther operator $\pi_H^V$, but it is not a projection yet. In order to modify $S^1$ accordingly, another operator $$ 
Q^1_E:
\mathbf{H}(\curl, \omega_E^{\mathit{ext}})
\to
\mathcal{N}(\mathcal{T}_H(\omega_E^{\mathit{ext}}))
$$
needs to be introduced for all edges $E\in \Delta_1(\TauH)$. Given  an edge $E$ and some some $\mathbf{u}\in \mathbf{H}(\curl, \omega_E^{\mathit{ext}})$ the image $Q^1_E(\bfu) \in \mathcal{N}(\mathcal{T}_H(\omega_E^{\mathit{ext}}))$ is defined by the system
\begin{equation*}
\begin{aligned}
(\mathbf{u}-Q^1_E \mathbf{u}, \nabla \tau) &= 0 \quad 
&&\text{for all } \tau\in \mathcal{S}(\mathcal{T}_H(\omega_E^{\mathit{ext}}))
\\
(\curl (\mathbf{u}-Q^1_E \mathbf{u}),\curl \mathbf{v}) &=0
&&\text{for all } \mathbf{v}\in \mathcal{N}(\mathcal{T}_H(\omega_E^{\mathit{ext}})).
\end{aligned}
\end{equation*}
Again, existence and uniqueness follow from the exactness of the discrete complex (cf. \cite{ArnoldFalkWinther2006,ArnoldFalkWinther2010,FalkWinther2014}).

With $Q^1_E$ available, it is now possible to compute the edge-based Falk-Winther projection according to \eqref{e:R1def} by
\begin{align*}
\pi_H^E \mathbf{u}
=
S^1 \mathbf{u}
+
\sum_{E\in\Delta_1(\TauH)}
\int_E \big((\id -S^1)Q^1_E \mathbf{u}\big)\cdot \mathbf{t}_E\,ds
\,\boldsymbol{\psi}_E ,
\end{align*}
where $I$ denotes the identity operator.

\begin{remark}
	Similarly to the $S^1$ operator in Remark~\ref{S_assemble}, we seek a matrix representation $Q_E$ of $Q^1_E: \mathcal{N}(\mathcal{T}_h) \rightarrow \mathcal{N}(\TauH)$ for the implementation. Given an edge $E\in \TauH$, assemble the matrices
	\begin{alignat*}{2}
	A_E &:= \left[\int_{\omega_E^{\mathit{ext}}} \curl \boldsymbol{\psi}_E \cdot \curl \boldsymbol{\psi}_{\hat E} \,dx\, \right]_{E,\hat E \in \Delta_1(\TauH(\omega_E^{\mathit{ext}}))} \quad && B_E:=\left[\int_{\omega_E^{\mathit{ext}}} \boldsymbol{\psi}_E \cdot \nabla \lambda_y \,dx\, \right]_{\begin{subarray}{l} E\in \Delta_1(\TauH(\omega_E^{\mathit{ext}})) \\y \in \Delta_0(\TauH(\omega_E^{\mathit{ext}})) \end{subarray}}\\
	C_E &:= \left[\int_{\omega_E^{\mathit{ext}}} \curl \boldsymbol{\psi}_E \cdot \curl \boldsymbol{\psi}_{e} \,dx\, \right]_{\begin{subarray}{l}E \in \Delta_1(\TauH(\omega_E^{\mathit{ext}})) \\ e \in \Delta_1(\mathcal{T}_h(\omega_E^{\mathit{ext}}))\end{subarray}}  && D_E:=\left[\int_{\omega_E^{\mathit{ext}}} \boldsymbol{\psi}_e \cdot \nabla \lambda_y \,dx\, \right]_{\begin{subarray}{l} e\in \Delta_1(\mathcal{T}_h(\omega_E^{\mathit{ext}})) \\y \in \Delta_0(\TauH(\omega_E^{\mathit{ext}})) \end{subarray}}
	\end{alignat*}
	and solve the following system for $Q_E$
	\begin{equation*}
	\begin{bmatrix}
	A_E & B^\ast_E \\ B_E & 0
	\end{bmatrix}
	\begin{bmatrix}
	Q_E \\ V_E
	\end{bmatrix}
	=
	\begin{bmatrix}
	C_E \\ D_E
	\end{bmatrix}.
	\end{equation*}
	This gives us $Q_E$. To achieve a matrix representation, here denoted by $P$, of the Falk-Winther projection $\pi^E_H:\mathcal{N}(\mathcal{T}_h) \rightarrow \mathcal{N}(\TauH)$, we restrict the matrix $S$ to edge patch $\omega_E^{\mathit{ext}}$ and denote this matrix by $S_{\omega_E^{\mathit{ext}}}$. Thus, $P$ can be computed by
	\begin{align*}
	P(E,[e_1,...,e_N]) = S(E,[e_1,...,e_N]) + (Q_E- S_{\omega_E^{\mathit{ext}}}Q_E)(E,:).
	\end{align*} 
	Note that $P$ can be assembled by looping over all the edges in the coarse mesh $\TauH$ and all computations are restricted to either the edge patch $\omega_E^{\mathit{ext}}$ or the nodal patch $\omega_y$. This means that the assembly of $P$ can be made efficiently.
\end{remark}

\end{document}